\documentclass[a4paper,11pt]{article}
\usepackage[abbrev]{amsrefs}
\usepackage[dvipdfmx]{hyperref}

\usepackage{latexsym,amssymb}
\usepackage[mathscr]{eucal}
\usepackage{color}
\usepackage{amsmath}
\usepackage{amssymb}
\usepackage{enumitem}
\usepackage{mathrsfs}
\usepackage{comment}

\usepackage{mathtools}
\usepackage{amsthm}

\newtheorem{Theorem}{\bf Theorem}[section]
\newtheorem{Proposition}{\bf Proposition}[section]
\newtheorem{Lemma}{\bf Lemma}[section]

\theoremstyle{definition}
\newtheorem{Remark}{\bf Remark}[section]

\newtheorem{Definition}{\bf Definition}[section]

\newcommand{\jn}{\quad\textrm{in}\quad}
\newcommand{\on}{\quad\textrm{on}\quad}

\newcommand{\jf}{\quad\textrm{if}\quad}

\newcommand{\FA}{\quad\textrm{for any}\quad}

\newcommand{\Ker}{\operatorname{Ker}}
\renewcommand{\Im}{\operatorname{Im}}

\newcommand{\supp}{\operatorname{supp}}
\newcommand{\loc}{{\mathrm{loc}}}
\newcommand{\R}{\mathbb{R}}
\newcommand{\D}{\mathcal{D}}

\numberwithin{equation}{section}
\allowdisplaybreaks[1]
\title{Supercritical Lane–Emden equation on a cone with an
inhomogeneous Dirichlet boundary condition
}
\author{
Sho Katayama\footnote{Email: \href{mailto:katayama-sho572@g.ecc.u-tokyo.ac.jp}{katayama-sho572@g.ecc.u-tokyo.ac.jp}}
\vspace{5pt}\\
Graduate School of Mathematical Sciences,\\
The University of Tokyo,\\
3-8-1 Komaba, Meguro-ku, Tokyo 153-8914, Japan
}
\date{}
\begin{document}
\maketitle
\renewcommand{\thefootnote}{\fnsymbol{footnote}}
\footnote[0]{2020AMS Subject Classification: 35J25, 35J61, 35B09, 35B32}
\renewcommand{\thefootnote}{\arabic{footnote}}
  \begin{abstract}
    We consider the Lane--Emden equation with a supercritical nonlinearity with an inhomogeneous Dirichlet boundary condition on an infinite cone. Under suitable conditions for the boundary data and the exponent of nonlinearity, we give a complete classification of the existence/nonexistence of a solution with respect to the size of boundary data. Moreover, we give a result on the multiple existence of solutions via bifurcation theory. We also state results on Hardy--H\'{e}non equations on infinite cones as a generalization.
  \end{abstract}
  \section{Introduction}
  This paper concerns a structure of solutions to a boundary value problem of the Lane--Emden equation
  \begin{equation}\label{L}\tag{\mbox{$\textrm{L}_\kappa$}}
    \left\{\begin{aligned}
      -\Delta u&=u^p\quad& &\textrm{in}\quad& &\Omega,\\
      u&>0\quad& &\textrm{in}\quad& &\Omega,\\
      u&=\kappa\mu& &\textrm{on}\quad& &\partial\Omega\setminus\{0\},
    \end{aligned}\right.
  \end{equation}
  where $p>1$, $\kappa>0$, and $\mu$ is a nonnegative non-zero continuous function on $\partial\Omega\setminus\{0\}$. Here, $\Omega$ is an infinite cone in $\R^N$, namely
  \begin{equation}\label{eq:cone}
  \Omega:=\{r\theta\in\R^N:r>0,\theta\in A\}
  \end{equation}  
  with $N\ge 2$ and $A$ being a subdomain of the unit sphere $S^{N-1}$ with a smooth boundary. A typical example of an infinite cone is the half-space $\R^N_+:=\{(x_1,\ldots,x_N)\in\R^N: x_N>0\}$.

 In this paper, under appropriate assumptions on the exponent of the nonlinearity $p$ and the boundary data $\mu$, we prove the existence/nonexistence and multiplicity of solutions to problem \eqref{L}. More precisely, our results state the existence of two constants $0\le\kappa_*<\kappa^*<\infty$ with the following properties.
\begin{enumerate}[label={\rm(\roman*)}]
  \item If $0<\kappa<\kappa^*$, then problem \eqref{L} possesses a solution.
  \item If $\kappa>\kappa^*$, then problem \eqref{L} possesses no solutions.
  \item If $\kappa=\kappa^*$, then problem \eqref{L} possesses a unique solution.
  \item If $\kappa_*<\kappa<\kappa^*$, then problem \eqref{L} possesses at least two solutions.
\end{enumerate}
  The existence result in the critical case (iii) and the multiplicity result (iv) in the Sobolev supercritical case $p>p_S$ are our main interests. Here, $p_S$ is the Sobolev critical exponent, namely
  \[
  p_S=\infty\jf N=2,\quad p_S:=\frac{N+2}{N-2}\jf N\ge 3.
  \]
 Existence and nonexistence of solutions to the boundary value problem for the Lane--Emden equation
  \begin{equation}\label{eq:SEBVP}
  \left\{\begin{aligned}
    -\Delta u&=u^p\quad & &\text{in}\quad& &\Omega,\\
    u&>0\quad& &\textrm{in}\quad& &\Omega,\\
    u&=\kappa\mu& &\textrm{on}\quad& &\partial\Omega,
    \end{aligned}\right.
  \end{equation}
  have been studied in various settings. Among others, Bidaut V\'eron and Vivier \cite{BV} proved that under assumptions that $\Omega$ is a bounded domain and $1<p<(N+1)/(N-1)$, problem \eqref{eq:SEBVP} possesses a solution for any nonnegative nontrivial Radon measure $\mu$ on $\partial\Omega$ and sufficiently small $\kappa$ (see also e.g., \cite{LX}). 
On the other hand, in the case $\Omega$ is an infinite cone \eqref{eq:cone}, Berestycki, Capuzzo--Dolcetta, and Nirenberg \cite{BCN} proved that if
  \begin{equation}\label{eq:pg}
  p\le p^*_\gamma:=\frac{N+\gamma}{N-2+\gamma},
  \end{equation}
there exists no function $u\in C^2(\Omega)$ such that
  \[
  -\Delta u\ge u^p\jn\Omega,\quad u>0\jn\Omega.
  \]
  Here, $\gamma$ is a positive constant determined as follows. Let $\Lambda>0$ be the first eigenvalue for the eigenvalue problem
  \begin{equation}\label{eq:Lambda}
    -\Delta'\psi=\Lambda\psi\jn A,\quad \psi=0\on \partial A,
  \end{equation}
  where $\Delta'$ is the Laplace--Beltrami operator on $S^{N-1}$. Then $\gamma$ is defined as the unique positive root of the equation
  \begin{equation}\label{eq:gamma}
    \gamma(N-2+\gamma)=\Lambda.
  \end{equation}
  (See also e.g., \cites{BE,BP,L,MP}.) We note that in the case $\Omega=\R^N_+$, since $A$ is the half-sphere $S^{N-1}_+=\{(\theta_1,\ldots,\theta_N)\in S^{N-1}:\ \theta_N>0\}$, the first eigenvalue $\Lambda$ equals $N-1$, and thus $\gamma=1$. Hence, this result implies that if 
  \[
  p\le p^*_1=\frac{N+1}{N-1},
  \]
  then there exists no positive function $u\in C^2(\R^N_+)$ such that $-\Delta u\ge u^p$ in $\R^N_+$. (See also e.g., \cite{GS}.)
  
  For the case $\Omega=\R^N_+$ and $p>p^*_1$, Fila, Ishige, and Kawakami \cites{FIK1,FIK3} proved that while problem \eqref{eq:SEBVP} possesses no solutions $u$ for any $\kappa>0$ if a nonnegative non-zero function $\mu\in L^\infty(\partial\R^N_+)$ satisfies
  \[
  \limsup_{|x|\to\infty}|x|^\frac{2}{p-1}\mu(x)=\infty,
  \]
  problem \eqref{eq:SEBVP} possesses a solution $u$ for sufficiently small $\kappa>0$ if
  \[
  \limsup_{|x|\to\infty}|x|^\frac{2}{p-1}\mu(x)<\infty.
  \]
  They also proved that the existence of a solution to problem \eqref{eq:SEBVP} is equivalent to the existence of a solution to the Lane--Emden equation with a dynamical boundary condition, namely
  \[
  \left\{\begin{aligned}
    -\Delta u&=u^p\quad & &\text{in}\quad& &\R^N_+\times(0,\infty),\\
    u&>0\quad& &\textrm{in}\quad& &\R^N_+\times(0,\infty),\\
    \partial_tu&=\partial_{x_N}u,\quad& &\textrm{on}\quad& &\partial\R^N_+\times(0,\infty),\\
    u(\cdot,0)&=\kappa\mu& &\textrm{on}\quad& &\partial\R^N_+.
    \end{aligned}\right.
  \]
  (See also \cites{FIK-1,FIK0}.) Independently, Bidaut--V\'eron, Hoang, Nguyen, and V\'eron \cite{BHNV} considered the solvability of problem~\eqref{eq:SEBVP} in the case $\Omega$ is either a bounded domain or the half-space. They gave a necessary and sufficient condition on a nonnegative nontrivial Radon measure $\mu$, admitting singularities, on $\partial\Omega$ for problem \eqref{eq:SEBVP} to possess a solution if $\kappa$ is small enough. Their results are described in words of capacity. Their method is an application of solvability theory for general nonlinear integral equations. (See e.g., \cite{KV}. See also e.g., \cite{FIK2} for other treatments of singular boundary data.) We note here that these results show that global structure of the domain $\Omega$ affects the existence/nonexistence of semilinear elliptic equations. Inspired by these studies, we aim to investigate how global structure of $\Omega$ affects the critical existence (iii) and multiplicity (iv).

  We next explain former studies concerning the existence in the critical case (iii) and the multiplicity (iv). Ai and Zhu \cite{AZ} proved the existence of a constant $\kappa^*>0$ with properties (i)--(iv) with $\kappa_*=0$ for a boundary value problem for a nonlinear scalar field equation
    \begin{equation}\label{eq:SFE}
  \left\{\begin{aligned}
    -\Delta u+u&=u^p\quad & &\text{in}\quad& &\R^N_+,\\
    u&>0\quad& &\textrm{in}\quad& &\R^N_+,\\
    u&=\kappa\mu& &\textrm{on}\quad& &\partial\R^N_+,
    \end{aligned}\right.
  \end{equation}
  with $1<p<p_S$ and $\mu\in H^{1/2}(\partial\R^N_+)\cap L^\infty(\partial\R^N_+)$. Subsequently, the author of this paper \cite{K} extended their result to supercritical case. He proved that under the condition $1<p<p_{JL}$ and an appropriate assumption on $\mu$, there exist constants $0\le\kappa_*<\kappa^*$ with properties (i)--(iv) for the problem \eqref{eq:SFE}. Here $p_{JL}$ is the so-called Joseph--Lundgren 
critical exponent, namely
\[
p_{JL}:=\infty\jf N\le 10,\quad p_{JL}:=1+\frac{4}{N-4-2\sqrt{N-1}}\jf N\ge 11.
\]
See also e.g., \cites{FW, IK, IOS01, NS01} and references therein for related results.

  Unfortunately, for the problem \eqref{eq:SEBVP} with $\Omega$ unbounded, no results concerning properties (iii) and (iv) are known even in the subcritical case $1<p<p_S$. We note that a lack of compactness conditions and a slow decay of the Green function for $-\Delta$ prevent us from applying methods in the aforementioned studies for problem~\eqref{eq:SFE}. The main difficulty arises in a possibility for solutions to be \textit{slow-decay}, that is
  \[
  \limsup_{|x|\to\infty}|x|^{\frac{2}{p-1}}u(x)>0,
  \]
  which makes it difficult to study the critical existence and the multiplicity of solutions (see Remark~\ref{Rmk:cptdecay}). 
  
  Ishige and the author of this paper \cite{IK} considered the Lane--Emden equation with a forcing term
  \begin{equation}\label{eq:Fterm}
  -\Delta u=u^p+\kappa\mu\jn\R^N,\quad u>0\jn\R^N,
  \end{equation}
  which involves similar difficulties concerning decays of solutions to that of problem \eqref{L}. They used the Kelvin transform to obtain decay estimates of solutions and find a sufficient condition for the exponent $p$ to obtain the property (iii). Although this method can be applied to problem \eqref{L}, it is not appropriate for our purpose since it cannot involve differences of the domains.
  
  In this paper, we introduce new decay estimates of solutions (see Lemmas~\ref{Lem:NRG2}~and~\ref{Lem:Reg2}) to obtain the fast-decay property of solutions. Furthermore, these calculations provide a new-type condition for the exponent $p$ (see Theorem~\ref{Thm:Thm2} and \eqref{eq:Hq}). We elucidate here that our new condition for assertions (iii) and (iv) is wider than that for the problem \eqref{eq:Fterm} (see e.g., \cite{IK}), and strongly depends on the shape of the domain $\Omega$. In this sense, our results reveal an aspect of how the Dirichlet boundary condition and global structure of the domain affect structures of solutions to semilinear elliptic equations.

  In what follows, we always denote by $\Lambda$, $\gamma$, and $p^*_\gamma$ the constants defined in \eqref{eq:Lambda}, \eqref{eq:gamma}, and \eqref{eq:pg}, respectively. We also fix constants $\alpha$ and $\beta$ so that
  \begin{equation}\label{eq:abcond}
      -N+2-\gamma<\beta<-\frac{2}{p-1}<\alpha<\gamma.
    \end{equation}
  To state our results, we clarify our concept of solutions to problem~\eqref{L}.
  \begin{Definition}\label{Def:sols}
    Let $p>p^*_\gamma$, $\kappa>0$, and $\mu\in C(\partial\Omega\setminus\{0\})$ be nonnegative and non-zero.
    \begin{enumerate}[label={\rm(\roman*)}]
      \item We call $u\in C^2(\Omega)\cap C(\overline{\Omega}\setminus\{0\})$ a solution to problem~\eqref{L} if and only if $u$ satisfies
      \begin{equation}\label{eq:solrestrict}
        \limsup_{|x|\to+0}|x|^{-\alpha}u(x)<\infty,\quad \limsup_{|x|\to\infty}|x|^{-\beta}u(x)<\infty,
      \end{equation}
and \eqref{L} in pointwise sense.
      \item We call $u\in C^2(\Omega)\cap C(\overline{\Omega}\setminus\{0\})$ a supersolution to problem~\eqref{L} if and only if $u$ satisfies \eqref{eq:solrestrict} and
      \[
      -\Delta u\ge u^p\jn\Omega,\quad u>0\jn\Omega,\quad u\ge \kappa\mu\on\partial\Omega.
      \]
      \item We call $u\in C^2(\Omega)\cap C(\overline{\Omega}\setminus\{0\})$ a minimal solution to problem~\eqref{L} if and only if $u$ is a solution to problem~\eqref{L} and satisfies $u\le v$ in $\Omega$ for all solutions $v$ to problem~\eqref{L}.
    \end{enumerate}
  \end{Definition}
   The condition \eqref{eq:solrestrict} is equivalent to $u\in C_{\alpha,\beta}$, where
  \begin{gather}
  C_{\alpha,\beta}:=\left\{f\in C(\overline{\Omega}\setminus\{0\}):\ \|f\|_{C_{\alpha,\beta}}:=\sup_{x\in\Omega}U_{\alpha,\beta}(x)^{-1}|f(x)|<\infty\right\},\label{eq:Cab}\\
  U_{\alpha,\beta}(x):=|x|^\alpha(1+|x|^2)^\frac{\beta-\alpha}{2}.\label{eq:Uab}
  \end{gather}
  Note that $C_{\alpha,\beta}\subset L^q(\Omega)$ if and only if $\alpha>-N/q$ and $\beta<-N/q$ for $q\in[1,\infty)$. We remark that a minimal solution to problem~\eqref{L} is unique. In what follows, we denote by $u^\kappa$ the minimal solution to problem~\eqref{L} if it exists.
  
  Now we are ready to state our results. Our first result concerns the existence of a threshold constant $\kappa^*$ for the existence/nonexistence of solutions to problem \eqref{L}. 
  \begin{Theorem}\label{Thm:Thm1}
    Assume that $p>p^*_\gamma$ and that a nonnegative non-zero function $\mu\in C(\partial\Omega\setminus\{0\})$ satisfies 
    \[
    \limsup_{|x|\to+0}|x|^{-\alpha}\mu(x)<\infty,\quad\limsup_{|x|\to\infty}|x|^{-\beta}\mu(x)<\infty.
    \]
    Then there exists a constant $\kappa^*>0$ with following properties.
    \begin{enumerate}[label={\rm(\roman*)}]
      \item If $0<\kappa<\kappa^*$, then problem \eqref{L} possesses a minimal solution.
      \item If $\kappa>\kappa^*$, then problem \eqref{L} possesses no solutions.
    \end{enumerate}
  \end{Theorem}
  
  Our second theorem gives a complete classification of 
the existence/nonexistence of solutions to problem~\eqref{L}, and provides a multiple existence of solutions to problem~\eqref{L}, under a suitable condition for the exponent $p$. We define a cubic polynomial $H(q)$ by
  \begin{equation}\label{eq:Hq}
  H(q):=-4\Lambda q^3+(N-2)(N-10)q^2-8(N-4)q+16.
  \end{equation}
  \begin{Theorem}\label{Thm:Thm2}
    Assume the same conditions as in Theorem~{\rm\ref{Thm:Thm1}}. Assume additionally that
    \begin{equation}\label{eq:pjl0}
    p<p_{JL},\quad H(p-1)<0.
    \end{equation}
    Then there exist constants $\kappa^*>0$ and $\kappa_*\in[0,\kappa^*)$ with following properties.
    \begin{enumerate}[label={\rm(\roman*)}]
      \item If $0<\kappa\le\kappa^*$, then problem~\eqref{L} possesses a minimal solution $u^\kappa$. Furthermore, if $\kappa=\kappa^*$, then $u^\kappa$ is the unique solution to problem~\eqref{L}.
      \item If $\kappa>\kappa^*$, then problem~\eqref{L} possesses no solutions.
      \item If $\kappa_*<\kappa<\kappa^*$, then problem~\eqref{L} possesses at least two solutions. In particular, there exists a solution $\overline{u}$ to problem~\eqref{L} such that $\overline{u}>u^\kappa$ in $\Omega$.
    \end{enumerate}
  \end{Theorem}
  The condition \eqref{eq:pjl0} can be rewritten in the form
  \[
  p\in(p_1,p_{JL}),
  \]
  where $p_1$ is a root of the cubic equation $H(p-1)=0$ with $p^*_\gamma<p_1<p_S$. We point out that it remains open whether problem \eqref{L} with $\kappa=\kappa^*$ possesses a solution $u$ when either $p\ge p_{JL}$ or $H(p-1)\ge 0$.

  \begin{Remark}
    A cubic equation $H(p-1)=0$ appears in a study by Dupaigne, Farina, and Petitt \cite{DFP}*{Corollary 2} in the context of stable solutions to the Lane--Emden equation on an infinite cone
  \[
    \left\{\begin{aligned}
      -\Delta u&=|u|^{p-1}u\quad& &\textrm{in}\quad& &\Omega,\\
      u&=0& &\textrm{on}\quad& &\partial\Omega,
    \end{aligned}\right.
  \]
  although this equation plays a different role from that in \cite{DFP} in this paper.
  \end{Remark} 

  Although our methods of proofs are inspired by \cite{IOS01}, we perform more delicate estimates in order to involve effects of the shape of $\Omega$. We first use a supersolution method (see Lemma~\ref{Lem:Supsol}) to prove that the set
    \begin{equation}\label{eq:Kappa}
  \mathcal{K}:=\{\kappa>0:\textrm{problem \eqref{L} possesses a solution}\}
  \end{equation}
  is of the form either $(0,\kappa^*]$ or $(0,\kappa^*)$, where
  \begin{equation}\label{eq:kappa*}
    \kappa^*:=\sup\mathcal{K}.
  \end{equation}
   Furthermore, we construct a supersolution to ensure the existence of a minimal solution to problem~\eqref{L} for small $\kappa>0$, which implies that $\kappa^*>0$ (see the proof of Proposition~\ref{Prop:Sec2}). In order to prove that $\kappa^*<\infty$, we consider a linearized eigenvalue problem
  \[
  -\Delta\phi=\lambda p(u^\kappa)^{p-1}\phi\jn\Omega,\quad \phi\in\D^{1,2}_0(\Omega),
  \]
  to prove the stability of minimal solutions, that is
  \[
  \int_\Omega p(u^\kappa)^{p-1}\psi^2\, dx\le\int_\Omega |\nabla\psi|^2\, dx\FA \psi\in\D^{1,2}_0(\Omega).
  \]
  Here, the function space $\D^{1,2}_0(\Omega)$ is the completion of $C^\infty_c(\Omega)$ with respect to the norm $\|\nabla\cdot\|_{L^2(\Omega)}$.

We use the stability of minimal solutions to give uniform boundedness and decay estimates for them and prove that $\mathcal{K}=(0,\kappa^*]$ under the conditions $1<p<p_{JL}$ and $H(p-1)<0$ (see Section~5). In order to obtain a decay estimate for $u^\kappa$, instead of the Kelvin transform method used in \cite{IK}, we pay attention to a difference of the constant in the Hardy inequality, namely
    \begin{equation}\label{eq:Hardy}
      \left(\left(\frac{N-2}{2}\right)^2+\Lambda\right)\int_\Omega|x|^{-2}\phi^2\, dx\le\int_\Omega|\nabla\phi|^2\, dx\quad\textrm{for any}\quad\phi\in\mathcal{D}^{1,2}_0(\Omega)
    \end{equation}
    (see e.g., \cite{N}*{Proposition 4.1}), while the constant in the left hand side is $((N-2)/2)^2$ in the case $\Omega=\R^N$. In order to involve this difference in estimates of minimal solutions, we introduce new weighted energy estimates for power-type weights (See Lemma~\ref{Lem:NRG2}). 

We finally use a bifurcation theorem (see e.g., \cite{CR1}) to prove the rest of Theorem~\ref{Thm:Thm2}, that is the uniqueness of a solution in the case $\kappa=\kappa^*$ and the multiplicity assertion (iii) (see Section~6).
\begin{Remark}\label{Rmk:cptdecay}
  In the proof of (iii) in Theorem~\ref{Thm:Thm2}, it is essential to show that the operator $h\mapsto G[p(u^\kappa)^{p-1}h]$ is compact on the space $C_{\alpha,\beta}$, where $G$ is the Green operator for $-\Delta$ on $\Omega$ with the Dirichlet condition (see Section~2 for the precise definition and a compactness property of $G$). This is the point where we need decay estimates for $u^\kappa$. Indeed, without decay estimates, it is possible that $\limsup\limits_{|x|\to\infty}|x|^2u^\kappa(x)^{p-1}>0$, which prevents us from obtaining the desired compactness (see Remark~\ref{Rmk:Gcpt}).
\end{Remark}

  The rest of this paper is organized as follows. In Section~2, we study some basic properties of linear Poisson equations on infinite cones as preliminaries. In Section~3, we prove the existence of a minimal solution to problem~\eqref{L} for small $\kappa>0$. In Section~4, we study linearized eigenvalue problems for minimal solutions to problem~\eqref{L}, and complete the proof of Theorem~\ref{Thm:Thm1}. In Section~5, we obtain uniform estimates for minimal solutions to problem~\eqref{L}, and prove the existence of a solution to problem~\eqref{L} with $\kappa=\kappa^*$. In Section~6, we complete the proof of Theorem~\ref{Thm:Thm2}. In Section~7, we generalize our theorems to H\'{e}non-type equations (see \eqref{H} in Section~7).

\section{Preliminaries}
  In what follows, we always denote by $U_{\alpha,\beta}$ and $C_{\alpha,\beta}$ the function and the space defined in \eqref{eq:Uab} and \eqref{eq:Cab}, respectively. In this section, we summarize some properties of linear Poisson equations on infinite cones.
    \begin{Proposition}\label{Prop:Linear}
    Let $-N+2-\gamma<\beta<\alpha<\gamma$, $\tau\in(0,1)$, $f\in C^{0,\tau}_{\rm{loc}}(\Omega)\cap C_{\alpha-2,\beta-2}$, and $\mu\in C(\partial\Omega\setminus\{0\})$ be such that $\sup\limits_{x\in\partial\Omega\setminus\{0\}}U_{\alpha,\beta}(x)^{-1}|\mu(x)|<\infty$. Then the linear elliptic problem
    \begin{equation}\label{eq:Linear}
      -\Delta v=f\quad\textrm{in}\quad \Omega,\quad v=\mu\quad\textrm{on}\quad\partial\Omega\setminus\{0\}
    \end{equation}
    possesses a unique solution $v\in C^2(\Omega)\cap C_{\alpha,\beta}$. Furthermore, the following properties hold.
    \begin{enumerate}[label={\rm(\roman*)}]
      \item there exists a constant $C$, depending on $\alpha$, $\beta$, $N$, and $\Lambda$, such that
      \begin{equation}\label{eq:LinEsti}
      \|v\|_{C_{\alpha,\beta}}\le C\left(\|f\|_{C_{\alpha-2,\beta-2}}+\sup_{\partial\Omega}U_{\alpha,\beta}(x)^{-1}|\mu(x)|\right).
      \end{equation}
      \item If $f\ge 0$ in $\Omega$ and $\mu\ge 0$ on $\partial\Omega\setminus\{0\}$, then $v\ge 0$ in $\Omega$.
    \end{enumerate}
  \end{Proposition}
  We construct a solution $v\in C^2(\Omega)\cap C_{\alpha,\beta}$ to problem \eqref{eq:Linear} by the Perron method. To this end, we first prove the following supersolution lemma and the maximum principle.
 \begin{Lemma}\label{Lem:LinearSupsol}
   Assume the same conditions as in Proposition~\rm{\ref{Prop:Linear}}. Then there exists a function $V_{\alpha,\beta}\in C^2(\Omega)\cap C_{\alpha,\beta}$ such that
   \[
   -\Delta V_{\alpha,\beta}\ge U_{\alpha-2,\beta-2}\quad\textrm{in}\quad\Omega,\quad V_{\alpha,\beta}\ge U_{\alpha,\beta}\quad\textrm{in}\quad\Omega.
   \]
 \end{Lemma}
 \begin{proof}
   We first observe that
   \begin{align*}
     &-\Delta U_{\alpha,\beta}(x)\\
     &=-\alpha(\alpha+N-2)|x|^{\alpha-2}(1+|x|^2)^\frac{\beta-\alpha}{2}-2\alpha(\beta-\alpha)|x|^{\alpha}(1+|x|^2)^\frac{\beta-\alpha-2}{2}\\
     &\qquad -|x|^{\alpha}\left((\beta-\alpha)N(1+|x|^2)^\frac{\beta-\alpha-2}{2}+(\beta-\alpha)(\beta-\alpha-2)|x|^2(1+|x|^2)^\frac{\beta-\alpha-4}{2}\right)\\
     &\ge -\alpha(\alpha+N-2)|x|^{\alpha-2}(1+|x|^2)^\frac{\beta-\alpha}{2}\\
     &\qquad -(\beta-\alpha)(2\alpha+N+(\beta-\alpha-2))|x|^\alpha(1+|x|^2)^\frac{\beta-\alpha-2}{2}\\
     &=-\alpha(\alpha+N-2)|x|^{\alpha-2}(1+|x|^2)^\frac{\beta-\alpha}{2}-(\beta-\alpha)(\beta+\alpha+N-2)|x|^\alpha(1+|x|^2)^\frac{\beta-\alpha-2}{2}
   \end{align*}
   for $x\in\Omega$. We used the inequality $(\beta-\alpha)(\beta-\alpha-2)>0$, which follows from $\beta-\alpha<0$, in the third line. This together with
   \[
   (\beta-\alpha)(\beta+\alpha+N-2)=\beta(\beta+N-2)-\alpha(\alpha+N-2)
   \]
   implies that,
   \[
     -\Delta U_{\alpha,\beta}(x)\ge -\alpha(\alpha+N-2)|x|^{\alpha-2}(1+|x|^2)^\frac{\beta-\alpha}{2}
   \]
   if $\alpha(\alpha+N-2)\ge\beta(\beta+N-2)$, and
   \begin{align*}
    &-\Delta U_{\alpha,\beta}(x)\\
    &\ge-\alpha(\alpha+N-2)|x|^{\alpha-2}(1+|x|^2)^\frac{\beta-\alpha}{2}-(\beta-\alpha)(\beta+\alpha+N-2)|x|^{\alpha-2}(1+|x|^2)^\frac{\beta-\alpha}{2}\\
    &=-\beta(\beta+N-2)|x|^{\alpha-2}(1+|x|^2)^\frac{\beta-\alpha}{2}
   \end{align*}
   otherwise. Hence, it follows that
   \[
   -\Delta U_{\alpha,\beta}(x)\ge-MU_{\alpha-2,\beta-2}(x)\quad\textrm{for}\quad x\in\Omega,\quad M:=\max\left\{\alpha(\alpha+N-2),\beta(\beta+N-2)\right\}.
   \]
   We note that $\gamma(N-2+\gamma)=\Lambda$ and $-N+2-\gamma<\beta<\alpha<\gamma$ imply that $M<\Lambda$.
   
   Since $M<\Lambda$, the elliptic problem
   \[
   -\Delta'\widetilde{\phi}=M\widetilde{\phi}+1\jn A,\quad \widetilde{\phi}=1\on \partial A,
   \]
 possesses a unique solution $\widetilde{\phi}\in C^2(A)\cap C(\overline{A})$.  Furthermore, since the comparison principle is available for the operator $-\Delta'-M$ on $A$, we obtain $\widetilde{\phi}\ge 1$ in $A$. We set
   \[
   V_{\alpha,\beta}(x):=U_{\alpha,\beta}(x)\widetilde{\phi}(\theta),\quad \theta:=\frac{x}{|x|}\in A.
   \]
   It clearly holds that $V_{\alpha,\beta}\in C^2(\Omega)\cap C_{\alpha,\beta}$ and $V_{\alpha,\beta}\ge U_{\alpha,\beta}$. We compute
   \[
     -\Delta V_{\alpha,\beta}(x)=-\Delta U_{\alpha,\beta}(x)\widetilde{\phi}(\theta)+|x|^{-2}U_{\alpha,\beta}(x)(M\widetilde{\phi}(\theta)+1)\ge U_{\alpha-2,\beta-2}(x)
   \]
   for $x\in\Omega$, and complete the proof of Lemma~\ref{Lem:LinearSupsol}.
 \end{proof}

\begin{Lemma}\label{eq:LinMax}
  If $v\in C_{\alpha,\beta}$ is a superharmonic function in Perron sense (see e.g., \cite{GT}*{Sections 2.8 and 6.3}) with $v\ge 0$ on $\partial\Omega\setminus\{0\}$, then $v\ge 0$ in $\Omega$.
\end{Lemma}
  \begin{proof}
     Let $\alpha'<\alpha$ and $\beta'>\beta$ be such that $\alpha'>\beta'$. Let $V_{\alpha',\beta'}$ be as in Lemma~\ref{Lem:LinearSupsol}, and set
    \[
    w:=\frac{v}{V_{\alpha,\beta}}.
    \]
    Then $w$ is a supersolution to the equation
   \begin{equation}\label{eq:Perrweq}
   -\Delta w-2\frac{\nabla V_{\alpha',\beta'}}{V_{\alpha',\beta'}}\cdot\nabla w+\frac{-\Delta V_{\alpha',\beta'}}{V_{\alpha',\beta'}}w=0\quad \textrm{in}\quad \Omega
   \end{equation}
   in Perron sense. Indeed, let $B\Subset\Omega$ be a ball and $\overline{w}$ be a solution to
   \[
  -\Delta\overline{w}-2\frac{\nabla V_{\alpha',\beta'}}{V_{\alpha',\beta'}}\cdot\nabla \overline{w}+\frac{-\Delta V_{\alpha',\beta'}}{V_{\alpha',\beta'}}\overline{w}=0\quad \textrm{in}\quad B,\quad \overline{w}=w\on\partial B.
   \]
  Then a direct calculation yields that $-\Delta(V_{\alpha',\beta'}\overline{w})=0$ in $B$ and that $V_{\alpha',\beta'}\overline{w}=v$ on $\partial B$, which together with the definition of Perron supersolutions imply that $V_{\alpha',\beta'}w=v\ge V_{\alpha',\beta'}\overline{w}$ in $B$. Thus $w$ is a Perron supersolution to the equation \eqref{eq:Perrweq}.
   
   Since $-\Delta V_{\alpha',\beta'}/V_{\alpha',\beta'}>0$ in $\Omega$, it follows from the maximum principle that
   \[
   \inf_{\substack{x\in\Omega,\\ r_1<|x|<r_2}}w(x)\ge \min\left\{\inf_{\Omega\cap \partial B(0,r)}w,\inf_{\Omega\cap \partial B(0,R)}w,0\right\}
   \]
   for all $0<r_1<r_2$. Furthermore, it follows from $v\in C_{\alpha,\beta}$ that
   \begin{gather*}
   w(x)\ge -C\frac{\|v\|_{C_{\alpha,\beta}}r_1^\alpha}{r_1^{\alpha'}}=-C\|v\|_{C_{\alpha,\beta}}r_1^{\alpha-\alpha'}\quad \textrm{for}\quad x\in \Omega\cap \partial B(0,r_1),\ r_1\in(0,1),\\
   w(x)\ge -C\frac{\|v\|_{C_{\alpha,\beta}}r_2^\beta}{r_2^{\beta'}}=-C\|v\|_{C_{\alpha,\beta}}r_2^{\beta-\beta'}\quad \textrm{for}\quad x\in \Omega\cap \partial B(0,r_2),\ r_2>1.
   \end{gather*}
   Thus
   \[
   \inf_{\substack{x\in\Omega,\\ r_1<|x|<r_2}}\frac{v(x)}{V_{{\alpha',\beta'}}(x)}\ge \min\left\{-C\|v\|_{C_{\alpha,\beta}}r_1^{\alpha-\alpha'},-C\|v\|_{C_{\alpha,\beta}}r_2^{\beta-\beta'}\right\}
   \]
   for $0<r_1<1<r_2$. Letting $r_1\to+0$ and $r_2\to\infty$, we obtain $v\ge 0$ in $\Omega$ and complete the proof of Lemma~\ref{eq:LinMax}.
  \end{proof}
 
 \begin{proof}[Proof of Proposition~\rm{\ref{Prop:Linear}}]
    Assertion (ii) and the uniqueness of a solution to problem \eqref{eq:Linear} are immediate consequences of Lemma~\eqref{eq:LinMax}.
   
   We construct a solution to problem~\eqref{eq:Linear}. Let $K:=\|f\|_{C_{\alpha-2,\beta-2}}+\sup\limits_{x\in\Omega}U_{\alpha,\beta}(x)^{-1}|\mu(x)|$, and $V_{\alpha,\beta}\in C_{\alpha,\beta}$ be as in Lemma~\ref{Lem:LinearSupsol}. Then
   \begin{gather}
   -\Delta(KV_{\alpha,\beta})\ge KU_{\alpha-2,\beta-2}\ge f\jn \Omega,\quad  KV_{\alpha,\beta}\ge KU_{\alpha,\beta}\ge\mu\on \Omega\setminus\{0\},\label{eq:PerrSup}\\
   -\Delta(-KV_{\alpha,\beta})\le -KU_{\alpha-2,\beta-2}\le f\jn \Omega,\quad  -KV_{\alpha,\beta}\le -KU_{\alpha,\beta}\le\mu\on \Omega\setminus\{0\}.\label{eq:PerrSub}
   \end{gather}
   Let $\mathcal{S}$ be a set consisting of functions $w\in C_{\alpha,\beta}$ which are subsolutions to the equation $-\Delta u=f$ in $\Omega$ in Perron sense (see e.g., \cite{GT}*{Sections 2.8 and 6.3}) and satisfy $w\le\mu$ on $\partial\Omega\setminus\{0\}$. By \eqref{eq:PerrSub}, $-KV_{\alpha,\beta}$ belongs to $\mathcal{S}$. Furthermore, by \eqref{eq:PerrSup}, $KV_{\alpha,\beta}-w$ is a Perron superharmonic function on $\Omega$ and satisfies $KV_{\alpha,\beta}-w\ge 0$ for all $w\in\mathcal{S}$. This together with Lemma~\ref{eq:LinMax} implies that $w\le KV_{\alpha,\beta}$ in $\Omega$. Hence the value
   \[
     v(x):=\sup_{w\in\mathcal{S}}w(x)
   \]
  exists for each $x\in\overline{\Omega}\setminus\{0\}$. By the same argument as in \cite{GT}*{Sections 2.8 and 6.3}, $v$ belongs to $C^2(\Omega)\cap C(\overline{\Omega}\setminus\{0\})$, and $v$ is a solution to problem \eqref{eq:Linear}. Furthermore, since
   \[
   -CKU_{\alpha,\beta}\le -KV_{\alpha,\beta}\le v\le KV_{\alpha,\beta}\le CKU_{\alpha,\beta}\quad \textrm{in}\quad \Omega,
   \]
   $v$ belongs to $C_{\alpha,\beta}$, and \eqref{eq:LinEsti} holds. Thus (i) follows, and the proof of Proposition~\ref{eq:Linear} is complete.
 \end{proof}
 Let $\alpha,\beta,\tau$ be as in Proposition~\ref{Prop:Linear} and define a linear operator $G:C^{0,\tau}_{\loc}(\Omega)\cap C_{\alpha-2,\beta-2}\to C^2(\Omega)\cap C_{\alpha,\beta}$ by letting $Gf$ be a unique solution $v\in C^2(\Omega)\cap C_{\alpha,\beta}$ to problem
\[
-\Delta v=f\jn\Omega,\quad v=0\on\partial\Omega\setminus\{0\}.
\]
By \eqref{eq:LinEsti} and the fact that $C^{0,\tau}_{\loc}(\Omega)\cap C_{\alpha-2,\beta-2}$ is a dense subspace of $C_{\alpha-2,\beta-2}$, $G$ is uniquely continued to a bounded linear operator from $C_{\alpha-2,\beta-2}$ to $C_{\alpha,\beta}$, still denoted by $G$.

\begin{Proposition}\label{Prop:G}
  Let $\alpha,\beta$ be as in Proposition~{\rm\ref{Prop:Linear}}.
\begin{enumerate}[label={\rm(\roman*)}]
  \item For any $f\in C_{\alpha,\beta}$, $Gf$ belongs to $C^{1,\sigma}_{\loc}(\Omega)$ for any $\sigma\in(0,1)$.
  \item Let $\alpha'<\alpha$, $\beta'>\beta$. Then $G$ is a compact operator from $C_{\alpha-2,\beta-2}$ to $C_{\alpha',\beta'}$.
  \item Assume that $\alpha>-(N-2)/2$ and $\beta<-(N-2)/2$. Then for any $f\in C_{\alpha-2,\beta-2}$, $Gf$ belongs to $\D^{1,2}_0(\Omega)$ and solves $-\Delta Gf=f$ in $\Omega$ in the weak sense.
\end{enumerate}
\end{Proposition}
\begin{proof}
  (i) is an immediate consequence of elliptic interior estimates and a density argument (see e.g., \cite{GT}*{Theorem 3.9}).
  
  We next prove (ii). It suffices to prove that for any sequence $\{f_n\}_{n=1}^\infty\subset C^{0,\tau}_{\loc}(\Omega)\cap C_{\alpha-2,\beta-2}$ which is bounded in the norm of $C_{\alpha-2,\beta-2}$, the sequence $\{Gf_n\}_{n=1}^\infty$ possesses a subsequence which converges in the norm of $C_{\alpha',\beta'}$.

  Since $v_n:=Gf_n$ satisfies $-\Delta v_n=f_n$ in $\Omega$ and $v_n=0$ on $\partial\Omega\setminus\{0\}$, $\{v_n\}_{n=1}^\infty$ is bounded in the norm of $C_{\alpha,\beta}$, and equicontinuous on $\overline{B(0,r_2)}\setminus B(0,r_1)$ for all $0<r_1<r_2$ (see e.g., \cite{GT}*{Theorem 4.5 and Section 6.3, Remark 3}). Hence the Ascoli-Arzel\'{a} theorem implies that $\{v_n\}_{n=1}^\infty$ possesses a subsequence $\{v_{n_k}\}_{k=1}^\infty$ which converges uniformly to some $v\in C(\overline{\Omega}\setminus\{0\})$ on $\overline{B(0,r_2)}\setminus B(0,r_1)$ for all $0<r_1<r_2$. 

  Since $\sup\limits_{k}\|v_{n_k}\|_{C_{\alpha,\beta}}<\infty$, $v$ belongs to $C_{\alpha,\beta}$. Furthermore, for all $0<r_1<1<r_2$,
  \begin{align*}
    &\limsup_{k\to\infty}\|v_{n_k}-v\|_{C_{\alpha',\beta'}}\\
    &\le \limsup_{k\to\infty}\sup_{|x|<r_1}U_{\alpha',\beta'}(x)^{-1}\left(|v_{n_k}(x)|+|v(x)|\right)+\limsup_{k\to\infty}\sup_{x_1\le|x|\le r_2}U_{\alpha',\beta'}(x)^{-1}|v_{n_k}(x)-v(x)|\\
    &\qquad+\limsup_{k\to\infty}\sup_{|x|>r_2}U_{\alpha',\beta'}(x)^{-1}\left(|v_{n_k}(x)|+|v(x)|\right)\\
    &\le \sup_k\sup_{|x|<r_1}\frac{U_{\alpha,\beta}(x)}{U_{\alpha',\beta'}(x)}\left(\|v_{n_k}\|_{C_{\alpha,\beta}}+\|v\|_{C_{\alpha,\beta}}\right)+\sup_k\sup_{|x|>r_2}\frac{U_{\alpha,\beta}(x)}{U_{\alpha',\beta'}(x)}\left(\|v_{n_k}\|_{C_{\alpha,\beta}}+\|v\|_{C_{\alpha,\beta}}\right)\\
    &\le C\left(r_1^{\alpha-\alpha'}+r_2^{\beta-\beta'}\right)\left(\sup_{k}\|v_{n_k}\|_{C_{\alpha,\beta}}+\|v\|_{C_{\alpha,\beta}}\right).
  \end{align*}
    Letting $r_1\to+0$ and $r_2\to\infty$, we deduce that $\limsup\limits_{k\to\infty}\|v_{n_k}-v\|_{C_{\alpha,\beta}}=0$.

  We next prove (iii). We first prove the case $f\in C^{0,\tau}_{\loc}\cap C_{\alpha-2,\beta-2}$. It follows from the assumptions $\alpha>-(N-2)/2$ and $\beta<-(N-2)/2$ that
  \[
  2+2(\alpha-2)>-N,\quad 2+(\beta-2)<-N.
  \]
  This together with the Hardy inequality \eqref{eq:Hardy} implies that
  \[
  \left|\int_\Omega f\varphi\,dx\right|\le \|f\|_{C_{\alpha,\beta}}\left(\int_\Omega|x|^2U_{\alpha-2,\beta-2}^2\,dx\right)^\frac{1}{2}\left(\int_\Omega |x|^{-2}\varphi^2\,dx\right)^\frac{1}{2}\le C\|f\|_{C_{\alpha,\beta}}\|\nabla\varphi\|_{L^2(\Omega)}
  \]
  for any $\varphi\in \D^{1,2}_0(\Omega)$. By the Lax--Milgram theorem, the problem
  \[
  -\Delta\overline{v}=f\jn\Omega,\quad \overline{v}\in\D^{1,2}_0(\Omega)
  \]
  possesses a unique solution $\overline{v}$, and 
  \begin{equation}\label{eq:vbarest}
    \|\nabla\overline{v}\|_{L^2(\Omega)}\le C\|f\|_{C_{\alpha,\beta}}.
  \end{equation}
  Furthermore, it follows from elliptic regularity theorems (see e.g., \cite{GT}*{Theorems 4.5, 8.22, and 8.27}) that $\overline{v}\in C^2(\Omega)\cap C(\overline{\Omega}\setminus\{0\})$. It also follows from a similar barrier argument to the proof of Proposition~\ref{Prop:Linear} (ii) that $\overline{v}\in C_{\alpha,\beta}$. Hence we deduce $Gf=\overline{v}$ by the uniqueness of a solution $v\in C^2(\Omega)\cap C_{\alpha,\beta}$ to problem \eqref{eq:Linear}, and obtain (iii). The general case follows by a density argument using \eqref{eq:vbarest}. The proof of Proposition~\ref{Prop:G} is complete.
  \end{proof}
  \begin{Remark}\label{Rmk:Gcpt}
    Assertion (ii) implies that the operator $h\mapsto G[Kh]$ is compact on $C_{\alpha,\beta}$ if $-N+2-\gamma<\beta<\alpha<\gamma$ and $K\in C_{-2+\delta,-2-\delta}$ for some $\delta>0$. On the other hand, $G$ is not compact from $C_{\alpha-2,\beta-2}$ to $C_{\alpha,\beta}$. In particular, the operator $h\mapsto G[Kh]$ may not be compact on $C_{\alpha,\beta}$ if either $\limsup\limits_{|x|\to+0}|x|^2K(x)>0$ or $\limsup\limits_{|x|\to\infty}|x|^2K(x)>0$.
  \end{Remark}
  \section{Existence of solutions for small $\kappa$}
  In what follows, we always assume the same conditions as in Theorem~\ref{Thm:Thm1}. Define also the set $\mathcal{K}$ and the number $\kappa^*$ as in \eqref{eq:Kappa} and \eqref{eq:kappa*}, respectively. In this section, we prove the following proposition.
  \begin{Proposition}\label{Prop:Sec2}
    Assume the same conditions as in Theorem~{\rm\ref{Thm:Thm1}}. Then $\mathcal{K}$ equals either $(0,\kappa^*)$ or $(0,\kappa^*]$. Furthermore, $\kappa^*>0$.
  \end{Proposition}
  We first confirm that supersolution methods are available for problem~\eqref{L}.
  \begin{Lemma}\label{Lem:Supsol}
    Assume the same conditions as in Theorem~{\rm\ref{Thm:Thm1}}. Then for all $\kappa>0$, the following assertions are equivalent.
    \begin{enumerate}[label={\rm(\roman*)}]
      \item $\kappa\in\mathcal{K}$.
      \item Problem~\eqref{L} possesses a minimal solution $u^\kappa$.
      \item Problem~\eqref{L} possesses a supersolution $v$.
    \end{enumerate}
    Furthermore, if $\kappa\in\mathcal{K}$ and $v$ is a supersolution to problem~\eqref{L}, then $u^\kappa\le v$ in $\Omega$.
  \end{Lemma}
  \begin{proof}
    It is clear that (ii) $\implies$ (i) $\implies$ (iii). We prove that (iii) $\implies$ (ii). Assume (iii). We define approximate solutions $\{u^\kappa_j\}_{j=-1}^\infty\subset C^2(\Omega)\cap C_{\alpha,\beta}$ to problem~\eqref{L} as follows. Set $u^\kappa_{-1}:\equiv 0$, and let $u^\kappa_{j+1}\in C^2(\Omega)\cap C_{\alpha,\beta}$ be a solution to the linear problem
    \begin{equation}\label{eq:ApSolDef}
    -\Delta u^\kappa_{j+1}=(u^\kappa_j)^p\quad\textrm{in}\quad\Omega,\quad u^\kappa_{j+1}=\kappa\mu\quad\textrm{on}\quad\partial\Omega\setminus\{0\}.
    \end{equation}
    We point out that problem \eqref{eq:ApSolDef} possesses a unique solution $u^\kappa_{j+1}\in C^2(\Omega)\cap C_{\alpha,\beta}$. Indeed, it holds that $(u^\kappa_j)^p\in C^{0,\tau}_{\rm{loc}}(\Omega)\cap C_{p\alpha,p\beta}$. Since $\alpha<\gamma$, $\beta>-N+2-\gamma$, and
    \begin{equation}\label{eq:paa2}
    p\alpha>\alpha-2,\quad p\beta<\beta-2
    \end{equation}
    by \eqref{eq:abcond}, it follows from Proposition~\ref{Prop:Linear} that problem \eqref{eq:ApSolDef} possesses a unique solution $u^\kappa_{j+1}\in C^2(\Omega)\cap C_{\alpha,\beta}$. 

    By the definition of supersolutions to problem \eqref{L} and Proposition~\ref{Prop:Linear} (ii), we deduce inductively that
    \[
    u^\kappa_{j}\le u^\kappa_{j+1}\le v\quad \textrm{in}\quad\Omega
    \]
    for all $j\in\{-1,0,1,\ldots\}$. This together with elliptic regularity theorems implies that $\{\nabla^2u^\kappa_j\}_j$ is bounded in $C^{0,\tau}(\Omega')$ for each $\Omega'\Subset\Omega$. Thus
    \[
    u^\kappa(x):=\lim_{j\to\infty}u^\kappa_j(x)
    \]
    exists and satisfies $u^\kappa(x)\le v(x)$ for each $x\in\Omega$. Moreover, $u^\kappa$ belongs to $C^2(\Omega)\cap C_{\alpha,\beta}$ and satisfies
    \[
    -\Delta u^\kappa=\lim_{j\to\infty}(u^\kappa_j)^p=(u^\kappa)^p\quad \textrm{in}\quad \Omega,\quad u^\kappa=\kappa\mu\quad \textrm{on}\quad \partial\Omega\setminus\{0\}.
    \]
    Thus $u^\kappa$ is a solution to problem~\eqref{L} and satisfies $u^\kappa\le v$ in $\Omega$ for any supersolution $v\in C^2(\Omega)\cap C_{\alpha,\beta}$ to problem~\eqref{L}. This implies that $u^\kappa$ is a minimal solution to problem~\eqref{L}, and the proof of Lemma~\ref{Lem:Supsol} is complete.
  \end{proof}
  Now we prove Proposition~\ref{Prop:Sec2}.
  \begin{proof}[Proof of Proposition~{\rm{\ref{Prop:Sec2}}}]
    We first prove that $\mathcal{K}$ equals either $(0,\kappa^*)$ or $(0,\kappa^*]$. It suffices to prove that $0<\kappa<\kappa'$, $\kappa'\in\mathcal{K}$ $\implies$ $\kappa\in\mathcal{K}$. Since
    \[
    -\Delta u^{\kappa'}=(u^{\kappa'})^p\quad\textrm{in}\quad\Omega,\quad u^{\kappa'}=\kappa'\mu\ge\kappa\mu\quad\textrm{on}\quad\partial\Omega\setminus\{0\},
    \]
    $u^{\kappa'}$ is a supersolution to problem~\eqref{L} if $\kappa'>\kappa$ and $\kappa'\in\mathcal{K}$. This together with Lemma~\ref{Lem:Supsol} implies that $\kappa\in\mathcal{K}$.

    It remains to prove that $\kappa^*>0$, which is equivalent to $\mathcal{K}\neq\emptyset$. Let $V_{\alpha,\beta}\in C^2(\Omega)\cap C_{\alpha,\beta}$ be as in Lemma~\ref{Lem:LinearSupsol}. It follows from $\alpha-2>p\alpha$ and $\beta-2<p\beta$ that
    \[
    -\Delta V_{\alpha,\beta}\ge CU_{p\alpha,p\beta}\ge CV_{\alpha,\beta}^p\quad\textrm{in}\quad\Omega.
    \]
    Taking sufficiently small $\delta>0$, we obtain
    \[
    -\Delta (\delta V_{\alpha,\beta})\ge C\delta V_{\alpha,\beta}^p\ge (\delta V_{\alpha,\beta})^p\quad\textrm{in}\quad\Omega,\quad \delta V_{\alpha,\beta}\ge \delta U_{\alpha,\beta}\ge C\delta\mu\quad\textrm{on}\quad\partial\Omega\setminus\{0\}.
    \]
    This together with Lemma~\ref{Lem:Supsol} implies that $C\delta\in\mathcal{K}$, and the proof of Proposition~\ref{Prop:Sec2} is complete.
  \end{proof}
  \section{Eigenvalue Problems}
  In this section, we study the eigenvalue problem
  \begin{equation}\label{E}\tag{$\mbox{E}_{\kappa}$}
  -\Delta\phi=\lambda p(u^\kappa)^{p-1}\phi\quad\textrm{in}\quad\Omega,\quad\phi\in\mathcal{D}^{1,2}_0(\Omega).
  \end{equation}
  By a similar argument to in \cite{NS01}*{Lemma B.2}, we obtain the following property.
  \begin{Lemma}\label{Lem:EVP}
    Assume the same conditions as in Theorem~{\rm\ref{Thm:Thm1}}. Then the eigenvalue problem \eqref{E} possesses a unique first eigenpair $(\lambda^\kappa,\phi^\kappa)\in\R\times\mathcal{D}^{1,2}_0(\Omega)$ with the following properties.
  \begin{equation}\label{eq:EVP}
    \begin{gathered}
      \lambda^\kappa>0,\quad \phi^\kappa>0\quad\textrm{in}\quad\Omega,\quad\int_\Omega|\nabla\phi^\kappa|^2\, dx=1,\\
      \left\{\phi\in\mathcal{D}^{1,2}_0(\Omega):\ -\Delta\phi=\lambda^\kappa p(u^\kappa)^{p-1}\phi\quad\textrm{in}\quad\Omega\right\}=\operatorname{span}\{\phi^\kappa\},\\
      \int_\Omega \lambda^\kappa p(u^\kappa)^{p-1}\phi^2\, dx\le\int_\Omega|\nabla\phi|^2\, dx\quad\textrm{for any}\quad \phi\in\mathcal{D}^{1,2}_0(\Omega).
    \end{gathered}
  \end{equation}
  \end{Lemma}
  In what follows, we denote by $\lambda^\kappa$ and $\phi^\kappa$ the unique pair satisfying \eqref{eq:EVP}. In this section, we prove the stability of minimal solutions, that is the following proposition.
  \begin{Proposition}\label{Prop:stabil}
  Assume the same conditions as in Theorem~{\rm\ref{Thm:Thm1}}.
 Let $\kappa\in(0,\kappa^*)$. Then $\lambda^\kappa>1$. In particular,
  \begin{equation}\label{eq:stab}
    \int_\Omega p(u^\kappa)^{p-1}\phi^2\, dx\le\int_\Omega |\nabla\phi|^2\, dx\quad\textrm{for any}\quad \phi\in\mathcal{D}^{1,2}_0(\Omega).
  \end{equation}
  \end{Proposition}
  To this end, we investigate properties of differences of approximate solutions $\{u^\kappa_j\}_j$, namely
  \begin{equation}\label{eq:ApSolDiffDef}
  v^\kappa_j:=u^\kappa_j-u^\kappa_{j-1}\quad\textrm{for}\quad j\in\{0,1,\ldots\}.
  \end{equation}
  It follows from the same argument as in \cite{IK}*{Lemma 2.3} that
\begin{equation}\label{eq:umonotone}
v^\kappa_j>0,\quad u^\kappa_j\le\left(\frac{\kappa}{\kappa'}\right)u^{\kappa'}_j<u^{\kappa'}_j,\quad v^\kappa_j\le\left(\frac{\kappa}{\kappa'}\right)^{(p-1)j+1}v^{\kappa'}_j<v^{\kappa'}_j,
\end{equation}
for all $j\in\{0,1,\ldots\}$ and $0<\kappa<\kappa'$. We fix $\alpha_*>\alpha$ and $\beta_*<\beta$ so that
  \begin{gather}
    0<\alpha_*<\gamma,\quad -N+2-\gamma<\beta_*<-N+2,\label{eq:a*b*} \\
    p\alpha+\alpha_*>-N,\quad p\beta+\beta_*<-N,\label{eq:a*b*2} \\
    \alpha+\alpha_*>-N+2,\quad \beta+\beta_*<-N+2,\label{eq:a*b*3}
    \end{gather}
  and set
  \begin{equation}\label{eq:ajbj}
    \alpha_j:=\min\left\{\alpha+j(2+(p-1)\alpha),\alpha_*\right\},\quad \beta_*:=\max\left\{\beta+j(2+(p-1)\beta),\beta_*\right\}.
  \end{equation}
  To verify the existence of $\alpha_*,\beta_*$ satisfying \eqref{eq:a*b*}, \eqref{eq:a*b*2}, and \eqref{eq:a*b*3}, we prove the following lemma.
  \begin{Lemma}\label{Lem:a*b*}
    Assume the same conditions as in Theorem~{\rm{\ref{Thm:Thm1}}}. Then
    \begin{gather}
      p\alpha+\gamma>-N,\quad \alpha+\gamma>-N+2,\label{eq:a*}\\
      p\beta-\gamma<-2,\quad \beta-\gamma<0.\label{eq:b*}
    \end{gather}
  \end{Lemma}
  Indeed, once Lemma~\ref{Lem:a*b*} is proved, we find $\alpha_*\in(0,\gamma)$ and $\beta_*\in(-N+2-\gamma,-N+2)$ such that
  \begin{gather*}
    p\alpha+\alpha_*>-N,\quad \alpha+\alpha_*>-N+2,\\
    p\beta-(-N+2-\beta_*)<-2,\quad \beta-(-N+2-\beta_*)<0,
  \end{gather*}
  which is equivalent to \eqref{eq:a*b*2} and \eqref{eq:a*b*3}.
  \begin{proof}[Proof of Lemma~{\rm\ref{Lem:a*b*}}]
    We first prove \eqref{eq:a*}. It follows from $p>p^*_\gamma$ that
    \[
    \gamma>\frac{2}{p-1}-N+2.
    \]
    This together with \eqref{eq:abcond} implies that
    \begin{gather*}
    p\alpha+\gamma>-\frac{2p}{p-1}+\gamma=-\frac{2}{p-1}+\gamma-2>-N,\\
    \alpha+\gamma>-\frac{2}{p-1}+\frac{2}{p-1}-N+2=-N+2.
    \end{gather*}
    
    We next prove \eqref{eq:b*}. It follows from \eqref{eq:abcond} and $\gamma>0$ that
    \begin{gather*}
    p\beta-\gamma<-\frac{2p}{p-1}<-2,\quad\beta-\gamma<-\frac{2}{p-1}-\gamma<0.
    \end{gather*}
    Thus \eqref{eq:b*} holds. The proof of Lemma~\ref{Lem:a*b*} is complete.
  \end{proof}
  We find $j_*\in\{0,1,\ldots\}$ such that
  \[
    \alpha_j=\alpha_*,\quad \beta_j=\beta_*,\quad\textrm{for}\quad j\in\{j_*,j_*+1,\ldots\}.
  \]
  \begin{Lemma}\label{Lem:vk}
    Assume the same conditions as in Theorem~\rm{\ref{Thm:Thm1}}. Then for all $j\in\{0,1,\ldots\}$, $v^\kappa_j$ belongs to $C_{\alpha_j,\beta_j}$. In particular, for all $j\in\{j_*,j_*+1,\ldots\}$, $v^\kappa_j$ belongs to $C_{\alpha_*,\beta_*}$.
  \end{Lemma}
  \begin{proof}
    The case $j=0$ follows from the fact that $v^\kappa_0$ satisfies
    \[
    -\Delta v^\kappa_0=(u^\kappa_0)^p\quad\textrm{in}\quad\Omega,\quad v^\kappa_0=0\quad\textrm{on}\quad\partial\Omega\setminus\{0\},
    \]
    and Proposition~\ref{Prop:Linear}.

    Let $j\in\{0,1,\ldots\}$ and assume that $v^\kappa_j\in C_{\alpha_j,\beta_j}$. By \eqref{eq:ApSolDef} and \eqref{eq:ApSolDiffDef}, $v^\kappa_{j+1}$ satisfies
    \[
    -\Delta v^\kappa_{j+1}=(u^\kappa_j)^p-(u^\kappa_{j-1})^p\quad\textrm{in}\quad\Omega,\quad v^\kappa_{j+1}=0\quad\textrm{on}\quad\partial\Omega\setminus\{0\}.
    \]
    It follows from
    \[
    0\le(u^\kappa_j)^p-(u^\kappa_{j-1})^p\le p(u^\kappa_j)^{p-1}v^\kappa_j
    \]
    that
    \[
    (u^\kappa_j)^p-(u^\kappa_{j-1})^p\in C_{(p-1)\alpha+\alpha_j,(p-1)\beta+\beta_j}.
    \]
    Furthermore, we observe from \eqref{eq:abcond} and \eqref{eq:ajbj} that
    \[
    (p-1)\alpha+\alpha_j\ge \alpha_{j+1}-2,\quad (p-1)\beta+\beta_j\le \beta_{j+1}-2,
    \]
    \[
    -N+2-\gamma<\beta_*\le \beta_{j+1}<\beta<\alpha<\alpha_{j+1}\le\alpha_*<\gamma.
    \]
    These together with Proposition~\ref{Prop:Linear} imply that $v^\kappa_j\in C_{\alpha_{j+1},\beta_{j+1}}$, and complete the proof of Lemma~\ref{Lem:vk}.
  \end{proof}
  We set
  \[
  v^\kappa:=v^\kappa_{j_*}\quad \textrm{for}\quad \kappa>0,\quad w^\kappa:=u^\kappa-u^\kappa_{j_*+1}\quad \textrm{for}\quad \kappa\in\mathcal{K}.
  \]
  \begin{Lemma}\label{Lem:wk}
   Assume the same conditions as in Theorem~{\rm\ref{Thm:Thm1}}. Let $\kappa,\kappa'\in\mathcal{K}$ be such that $0<\kappa<\kappa'$. Then the following properties hold.
  \begin{enumerate}[label={\rm(\roman*)}]
    \item $0<w^\kappa<w^{\kappa'}$ in $\Omega$.
    \item $w^\kappa$ belongs to $C_{\alpha_*,\beta_*}$.
    \item $w^\kappa$ belongs to $\mathcal{D}^{1,2}_0(\Omega)$ and satisfies $-\Delta w^\kappa=(u^\kappa)^p-(u^\kappa_{j_*})^p$ in $\Omega$ in the weak sense.
  \end{enumerate}             
  \end{Lemma}
  \begin{proof}
    The relation (i) follows from \eqref{eq:umonotone} and the fact that
    \[
    w^\kappa=\lim_{j\to\infty}u^\kappa_j-u^\kappa_{j_*+1}=\sum_{j=j_*+2}^\infty v^\kappa_j.
    \]
    We next prove (ii). It is clear that $w^\kappa\in C^2(\Omega)\cap C_{\alpha,\beta}$ and
    \begin{equation}\label{eq:wprob}
      -\Delta w^\kappa=(u^\kappa)^p-(u^\kappa_{j_*})^p\quad\textrm{in}\quad\Omega,\quad w^\kappa=0\quad\textrm{on}\quad\partial\Omega\setminus\{0\}.
    \end{equation}
    By a similar induction to the proof of Lemma~\ref{Lem:vk}, we obtain $w^\kappa\in C_{\alpha_j,\beta_j}$ for all $j\in\{0,1,\ldots\}$, and thus $w^\kappa\in C_{\alpha_*,\beta_*}$.
    
    We finally prove (iii). Let a linear operator $G$ be as in Section~2. Then $w^\kappa=G[(u^\kappa)^p-(u^\kappa_{j_*})^p]$. It follows from \eqref{eq:abcond}, \eqref{eq:a*b*}, and the same argument as in the proof of Lemma~\ref{Lem:vk} that
    \[
    (u^\kappa)^p-(u^\kappa_{j_*})^p\in C_{(p-1)\alpha+\alpha_*,(p-1)\beta+\beta_*},
    \]
    \begin{equation}\label{eq:D12check}
    (p-1)\alpha+\alpha_*>-2\ge-\frac{N-2}{2}-2,\ (p-1)\beta+\beta_*<-N\le -\frac{N-2}{2}-2.
    \end{equation}
    These together with \eqref{Prop:G}~(iii) imply that $w^\kappa\in\D^{1,2}_0(\Omega)$ and complete the proof of Lemma~\ref{Lem:wk}.
  \end{proof}
  \begin{proof}[Proof of Proposition~{\rm\ref{Prop:stabil}}]
    By $w^\kappa\in\mathcal{D}^{1,2}_0(\Omega)$ and $0<w^\kappa<w^{\kappa'}$ for $\kappa,\kappa'\in\mathcal{K}$ with $\kappa<\kappa'$, Proposition~\ref{Prop:stabil} follows from the same argument as in \cite{IOS01}*{Lemma 4.6}.
  \end{proof}
  Now we are ready to prove Theorem~\ref{Thm:Thm1}.
  \begin{proof}[Proof of Theorem~{\rm\ref{Thm:Thm1}}]
  It remains to prove that $\kappa^*<\infty$.
    Let $\kappa\in(0,\kappa^*)$. Since $u^\kappa\ge u^\kappa_0=\kappa u^1_0>0$ in $\Omega$, it follows from \eqref{eq:stab} that
    \[
\kappa^{p-1}\int_\Omega p(u^1_0)^{p-1}\phi^2\, dx\le\int_\Omega|\nabla\phi|^2\, dx\quad\textrm{for all}\quad \phi\in\mathcal{D}^{1,2}_0(\Omega),
    \]
which implies that $\kappa^*<\infty$.
  \end{proof}
  In the end of this section, we prepare the following properties of eigenvalue functions $\phi^\kappa$, which plays a crucial role in the proof of the multiplicity assertion (iv) in Theorem~\ref{Thm:Thm2}.  Set
  \begin{equation}\label{eq:OmR}
  \Omega_R:=\{x\in\Omega:\ R<|x|<2R\},\quad \widetilde{\Omega}_R:=\left\{x\in\Omega:\ \frac{R}{2}<|x|<4R\right\},
  \end{equation}
  for $R>0$.
  \begin{Lemma}\label{Lem:phiimp}
    Assume the same conditions as in Theorem~{\rm\ref{Thm:Thm1}}. Then for all $\kappa\in\mathcal{K}$, $\phi^\kappa$ satisfies
    \begin{equation}\label{eq:phiG}
      \phi^\kappa\in C_{\alpha_*,\beta_*},\quad \phi^\kappa=\lambda^\kappa G[p(u^\kappa)^{p-1}\phi^\kappa],
    \end{equation}
    where $G$ is as in Section~{\rm 2}. Furthermore, $\phi^\kappa$ satisfies
    \begin{equation}\label{eq:philocN}
      \int_{\Omega_R}|\nabla\phi^\kappa|^2\, dx\le CR^{2\alpha_*+N-2}(1+R^2)^{\beta_*-\alpha_*}
    \end{equation}
    for all $R>0$.
  \end{Lemma}
  \begin{proof}

    We first prove that $\phi^\kappa=\lambda^\kappa G[p(u^\kappa)^{p-1}\phi^\kappa]$. It follows from elliptic regularity theorems (see e.g., \cite{GT}*{Theorems 4.5, 8.17, 8.22, and 8.27}) and the Hardy inequality \eqref{eq:Hardy} that $\phi^\kappa\in C^{2,\tau}_{\loc}(\Omega)\cap C(\overline{\Omega}\setminus\{0\})$ and that
    \[
    \|\phi^\kappa\|_{L^\infty(\Omega_R)}^2\le CR^{-N}\int_{\widetilde{\Omega}_R} (\phi^\kappa)^2\,dx\le CR^{-N+2}\int_{\overline{\Omega}_R} |x|^{-2}(\phi^\kappa)^2\,dx\le CR^{-N+2}
    \]
    for all $R>0$. Hence
    \[
    \phi^\kappa\in C^{2,\tau}_\loc(\Omega)\cap C_{-\frac{N-2}{2},-\frac{N-2}{2}}.
    \]
    Since
    \[
    -\frac{N-2}{2}+(p-1)\alpha>-\frac{N-2}{2}-2,\quad -\frac{N-2}{2}+(p-1)\beta<-\frac{N-2}{2}-2
    \]
    by \eqref{eq:abcond}, it follows from Proposition~\ref{Prop:G}~(iii) that $\overline{\phi}:=G[\lambda^\kappa p(u^\kappa)^{p-1}\phi^\kappa]$ is a solution to problem $\overline{\phi}=\lambda^\kappa p(u^\kappa)^{p-1}\phi^\kappa$ in $\Omega$, $\overline{\phi}\in\D^{1,2}_0(\Omega)$. This together with the uniqueness of a weak solution implies that $\phi^\kappa=\lambda^\kappa G[ p(u^\kappa)^{p-1}\phi^\kappa]$.

    It follows from the same induction as in the proof of Lemma~\ref{Lem:wk}~(ii) that $\phi^\kappa\in C_{\alpha_*,\beta_*}$. We finally deduce \eqref{eq:philocN} from a standard local energy estimate
    \begin{align*}
    \int_{\Omega_R}|\nabla\phi^\kappa|^2\, dx&\le \int_{\widetilde{\Omega}_R}\lambda^\kappa p(u^\kappa)^{p-1}(\phi^{\kappa})^2\, dx+CR^{-2}\int_{\widetilde{\Omega}_R}(\phi^{\kappa})^2\, dx\\
    &\le CR^{N+2\alpha_*}(1+R^2)^{\beta_*-\alpha_*}\left(R^{(p-1)\alpha}(1+R^2)^\frac{(p-1)(\beta-\alpha)}{2}+R^{-2}\right)
    \end{align*}
   and \eqref{eq:abcond}, and complete the proof of Lemma~\ref{Lem:phiimp}.
  \end{proof}
  \section{Uniform estimates of $\{w^\kappa\}_{\kappa\in(0,\kappa^*)}$}
  In this section, we obtain uniform estimates of $\{w^\kappa\}_{\kappa\in(0,\kappa^*)}$ to prove the following proposition.
  \begin{Proposition}\label{Prop:Sec4}
    Assume the same conditions as in Theorem~{\rm\ref{Thm:Thm2}}. Then $\kappa^*\in\mathcal{K}$.
  \end{Proposition}
 We start with giving a uniform energy estimate of $(w^\kappa)^\nu$, $\nu>1$, to obtain the following estimate. Let $\Omega_R$ and $\widetilde{\Omega}_R$ be as in \eqref{eq:OmR}.
  \begin{Lemma}\label{Lem:NRG1}
    Assume the same conditions as in Theorem~{\rm\ref{Thm:Thm1}}. Let $\nu\ge 1$ and $q$ be such that
    \begin{gather}
      \frac{\nu^2}{2\nu-1}<p,\label{eq:nucond}\\
     q\ge1\jf N=2,\quad 1\le q\le\frac{N}{N-2}\jf N\ge 3.\label{eq:qcond}
    \end{gather}
    Then
    \begin{equation}\label{eq:NRG1}
      \sup_{\substack{\kappa\in(0,\kappa^*),\\R>0}}R^{-\frac{N}{q}+\frac{4\nu}{p-1}}\left(\int_{\Omega_R}(w^\kappa)^{2q\nu}\, dx\right)^\frac{1}{q}<\infty.
    \end{equation}
  \end{Lemma}
  \begin{proof}
    Let $\eta\in C^\infty_c(0,\infty)$ be such that
    \[
    \operatorname{supp}\eta\subset\left(\frac{1}{2},4\right),\quad 0\le\eta(t)\le 1\quad\textrm{for}\quad t>0,\quad \eta(t)=1\quad\textrm{for}\quad t\in(1,2),
    \]
    and define $\eta_R,\zeta_R\in C^\infty_c(\R^N)$ by
    \[
    \eta_R(x):=\eta\left(\frac{|x|}{R}\right),\quad\zeta_R(x):=\eta_R(x)^m,
    \]
    for $R>0$, where $m$ is an integer such that $m>1+\frac{p-1}{2\nu}$. Then $\zeta_R$ satisfies
    \begin{gather}
    \begin{gathered}
      \supp\zeta_R\subset \widetilde{\Omega}_R, \quad 0\le \zeta_R\le 1\quad\textrm{in}\quad\R^N,\quad\zeta_R=1\quad\textrm{on}\quad \Omega_R,
      \end{gathered}\label{eq:zetaR1}\\
      \begin{gathered}
      |\nabla\zeta_R|\le m\eta_R^{m-1}|\nabla\eta_R|\le CR^{-1}\zeta_R^{1-\frac{1}{m}}\le CR^{-1}\zeta_R^{\frac{p-1}{p-1+2\nu}},\\
      \zeta_R|\nabla^2\zeta_R|\le C\left(\eta_R^{2m-1}|\nabla^2\eta_R|+\eta_R^{2m-2}|\nabla\eta_R|^2\right)\le CR^{-2}\zeta_R^{2-\frac{2}{m}}\le CR^{-2}\zeta_R^\frac{2(p-1)}{p-1+2\nu}.
        \end{gathered}\label{eq:zetaR2} 
    \end{gather}
    We first observe that
    \begin{equation}\label{eq:NRG1-1}
      \begin{aligned}
        &\int_\Omega |\nabla(\zeta_R(w^\kappa)^\nu)|^2\, dx\\
        &\begin{multlined}=\int_\Omega\left(\frac{\nu^2}{2\nu-1}\nabla w^\kappa\cdot\nabla\left(\zeta_R^2 (w^\kappa)^{2\nu-1}\right)\right.\\
        \left.\qquad+2\left(\nu-\frac{\nu^2}{2\nu-1}\right)\zeta_R\nabla\zeta_R\cdot(w^\kappa)^{2\nu-1}\nabla w^\kappa+|\nabla\zeta_R|^2(w^\kappa)^{2\nu}\right)\, dx\end{multlined}\\
        &=\int_\Omega\left(\frac{\nu^2}{2\nu-1}\nabla w^\kappa\cdot\nabla(\zeta_R^2(w^\kappa)^{2\nu-1})+\left(|\nabla\zeta_R|^2-\frac{\nu-1}{2\nu-1}\operatorname{div}(\zeta_R\nabla\zeta_R)\right)(w^\kappa)^{2\nu}\right)\, dx\\
        &=\int_\Omega\left(\frac{\nu^2}{2\nu-1}\nabla w^\kappa\cdot\nabla(\zeta_R^2(w^\kappa)^{2\nu-1})+\left(\frac{\nu}{2\nu-1}|\nabla\zeta_R|^2-\frac{\nu-1}{2\nu-1}\zeta_R\Delta\zeta_R\right)(w^\kappa)^{2\nu}\right)\, dx.
      \end{aligned}
    \end{equation}
    Let $\varepsilon,\delta\in(0,1)$. It follows from \eqref{eq:wprob} and \cite{IOS01}*{Lemma 5.1} that
    \begin{equation}\label{eq:NRG1-2}
      \begin{aligned}
      &\int_\Omega\nabla w^\kappa\cdot\nabla(\zeta_R^2(w^\kappa)^{2\nu-1})\, dx\\
      &=\int_\Omega ((u^\kappa)^p-(u^\kappa_{j_*})^p)\zeta_R^2(w^\kappa)^{2\nu-1}\, dx\\
      &=\int_\Omega \left((1+\varepsilon)(u^\kappa)^{p-1}(w^\kappa+v^\kappa)+C(u^\kappa_{j_*})^{p-1+2\nu\delta}(w^\kappa+v^\kappa)^{1-2\nu\delta}\right)\zeta_R^2(w^\kappa)^{2\nu-1}\, dx\\
      &\le \int_\Omega \left((u^\kappa)^{p-1}\zeta_R^2\left((1+\varepsilon)^2(w^\kappa)^{2\nu}+C(v^\kappa)^{2\nu}\right)+C(u^\kappa_{j_*})^{p-1+2\nu\delta}\zeta_R^2\left(w^\kappa+v^\kappa\right)^{2\nu(1-\delta)}\right)\, dx.
      \end{aligned}
    \end{equation}
    By \eqref{eq:abcond}, \eqref{eq:stab}, \eqref{eq:umonotone}, \eqref{eq:zetaR1}, and $v^{\kappa^*}\in C_{\alpha_*,\beta_*}$,
    \begin{align}
      &\int_\Omega (u^\kappa)^{p-1}\zeta_R^2(w^\kappa)^{2\nu}\, dx\le \frac{1}{p}\int_\Omega|\nabla(\zeta_R(w^\kappa)^\nu)|^2\, dx,\label{eq:NRG1-3}\\
      \begin{split}
      &\int_\Omega (u^\kappa)^{p-1}\zeta_R^2(v^{\kappa^*})^{2\nu}\, dx\le\sup_{\widetilde{\Omega}_{R}}(v^{\kappa^*})^{2\nu}\int_\Omega (u^\kappa)^{p-1}\zeta_R^2\, dx\\
      &\le CR^{2\nu\alpha_*}(1+R^2)^{\nu(\beta_*-\alpha_*)}\int_\Omega|\nabla\zeta_R|^2\,dx\le CR^{s_1}(1+R^2)^\frac{s_2-s_1}{2},
  \end{split}\label{eq:NRG1-4}
    \end{align}
    where
    \begin{gather*}
      s_1:=N-2+2\nu\alpha_*>N-2+2\nu\alpha>N-2-\frac{4\nu}{p-1},\\
      s_2:=N-2+2\nu\beta_*<N-2+2\nu\beta<N-2-\frac{4\nu}{p-1}.
    \end{gather*}
    Furthermore, for $\delta\in(0,1)$, we observe from the H\"{o}lder inequality, the Poincar\'{e} inequality, $u^{\kappa^*}_{j_*}\in C_{\alpha,\beta}$, $v^{\kappa^*}\in C_{\alpha_*,\beta_*}$, and \eqref{eq:zetaR1} that
    \begin{align}
      \begin{split}
      &\int_\Omega \zeta_R^2(u^\kappa_{j_*})^{p-1+2\nu\delta}(w^\kappa)^{2\nu(1-\delta)}\, dx\\
      &\le \sup_{\widetilde{\Omega}_R}(u^{\kappa^*}_{j_*})^{p-1+2\nu\delta}\left(\int_\Omega\zeta_R^2\,dx\right)^{\delta}\left(\int_{\widetilde{\Omega}_R}\zeta_R^2w^{2\nu}\,dx\right)^{1-\delta}\\
      &\le C(R^\alpha(1+R^2)^\frac{\beta-\alpha}{2})^{p-1+2\nu\delta}R^{N\delta}\left(R^2\int_{\widetilde{\Omega}_R}|\nabla\zeta_Rw^\nu|^2\,dx\right)^{1-\delta}\\
      &\le \varepsilon\int_{\widetilde{\Omega_R}}|\nabla\zeta_Rw^\nu|^2\,dx+C\left(R^{N\delta+2(1-\delta)}(R^\alpha(1+R^2)^\frac{\beta-\alpha}{2})^{p-1+2\nu\delta}\right)^\frac{1}{\delta}\\
      &=\varepsilon\int_\Omega|\nabla(\zeta_R(w^\kappa)^\nu|^2\, dx+CR^{s_3}(1+R^2)^{s_4-s_3},
      \end{split}\label{eq:NRG1-5}\\
        &\int_\Omega \zeta_R^2(u^{\kappa}_{j_*})^{p-1+2\nu\delta}(v^\kappa)^{2\nu(1-\delta)}\, dx\le CR^N\sup_{\widetilde{\Omega}_R}U_{\alpha,\beta}^{p-1+2\nu\delta}U_{\alpha_*,\beta_*}^{2\nu(1-\delta)}\le CR^{s_5}(1+R^2)^\frac{s_6-s_5}{2},\label{eq:NRG1-6}
    \end{align}
    where
    \begin{gather*}
    s_3:=\frac{1}{\delta}\left(2(1-\delta)+(p-1)\alpha\right)+N+2\nu\alpha,\\
    s_4:=\frac{1}{\delta}\left(2(1-\delta)+(p-1)\beta\right)+N+2\nu\beta,\\
    s_5:=N+(p-1+2\nu\delta)\alpha+2\nu(1-\delta)\alpha_*,\quad s_6:=N+(p-1+2\nu\delta)\beta+2\nu(1-\delta)\beta_*.
    \end{gather*}
    Since $(p-1)\alpha_*>(p-1)\alpha>-2>(p-1)\beta>(p-1)\beta_*$ by \eqref{eq:abcond} and \eqref{eq:a*b*}, 
    \begin{gather*}
    \lim_{\delta\to+0}s_3=\infty,\quad \lim_{\delta\to +0}s_4=-\infty,\\
    \lim_{\delta\to+0}s_5=N+(p-1)\alpha+2\nu\alpha_*>N-2-\frac{4\nu}{p-1},\\
    \lim_{\delta\to+0}s_6=N+(p-1)\beta+2\nu\beta_*<N-2-\frac{4\nu}{p-1}.
  \end{gather*}
  Thus, taking sufficiently small $\delta\in(0,1)$, we have
  \[
  s_3,s_5>N-2-\frac{4\nu}{p-1},\quad s_4,s_6<N-2-\frac{4\nu}{p-1}.  
  \]
  These together with \eqref{eq:NRG1-2}--\eqref{eq:NRG1-6} imply that
  \begin{equation}\label{eq:NRG1-6.1}
  \begin{aligned}
    &\int_\Omega \nabla w^\kappa\cdot\nabla(\zeta_R^2(w^\kappa)^{2\nu-1})\, dx\\
    &\begin{multlined}
    \le\frac{(1+\varepsilon)^2}{p}\int_\Omega|\nabla(\zeta_R(w^\kappa)^\nu)|^2\, dx\\
    +CR^{s_1}(1+R^2)^\frac{s_2-s_1}{2}+CR^{s_3}(1+R^2)^\frac{s_4-s_3}{2}+CR^{s_5}(1+R^2)^\frac{s_6-s_5}{2}
      \end{multlined}\\
    &\le\frac{(1+\varepsilon)^2}{p}\int_\Omega|\nabla(\zeta_R(w^\kappa)^\nu)|^2\, dx+CR^{N-2-\frac{4\nu}{p-1}}.
  \end{aligned}
  \end{equation}
  On the other hand, it follows from \eqref{eq:stab} and \eqref{eq:zetaR2} that for any $\varepsilon'>0$,
  \begin{equation}\label{eq:NRG1-7}
    \begin{aligned}
    \int_\Omega\left(\frac{\nu}{2\nu-1}|\nabla\zeta_R|^2-\frac{\nu-1}{2\nu-1}\zeta_R\Delta\zeta_R\right)&(w^\kappa)^{2\nu}\, dx
    \le C\int_\Omega R^{-2}\zeta_R^\frac{2\nu}{p-1+2\nu}(w^\kappa)^{2\nu}\, dx\\
    &\le CR^{-2+\frac{N(p-1)}{p-1+2\nu}}\left(\int_\Omega \zeta_R^2(w^\kappa)^{p-1+2\nu}\, dx\right)^\frac{2\nu}{p-1+2\nu}\\
    &\le CR^{-2+\frac{N(p-1)}{p-1+2\nu}}\left(\int_\Omega (u^\kappa)^{p-1}\zeta_R^2(w^\kappa)^{2\nu}\, dx\right)^\frac{2\nu}{p-1+2\nu}\\
    &\le CR^{-2+\frac{N(p-1)}{p-1+2\nu}}\left(\int_\Omega |\nabla(\zeta_R(w^\kappa)^\nu)|^2\, dx\right)^\frac{2\nu}{p-1+2\nu}\\
    &\le \varepsilon'\int_\Omega |\nabla(\zeta_R(w^\kappa)^\nu)|^2\, dx+CR^{\frac{-2(p-1+2\nu)+N(p-1)}{p-1}}\\
    &=\varepsilon'\int_\Omega |\nabla(\zeta_R(w^\kappa)^\nu)|^2\, dx+CR^{N-2-\frac{4\nu}{p-1}}.
    \end{aligned}
  \end{equation}
  Combining \eqref{eq:NRG1-1}, \eqref{eq:NRG1-6.1}, and \eqref{eq:NRG1-7}, we deduce that
  \[
  \int_\Omega |\nabla(\zeta_R(w^\kappa)^\nu)|^2\, dx\le\left(\frac{(1+\varepsilon)^2}{p}\frac{\nu^2}{2\nu-1}+C_\varepsilon\varepsilon'\right)\int_\Omega |\nabla(\zeta_R(w^\kappa)^\nu)|^2\, dx+C_{\varepsilon,\varepsilon'}R^{N-2-\frac{4\nu}{p-1}}
  \]
  for any $\varepsilon,\varepsilon'\in(0,1)$. By \eqref{eq:nucond} and a suitable choice of $\varepsilon$ and $\varepsilon'$, we obtain
  \[
  \int_\Omega |\nabla(\zeta_R(w^\kappa)^\nu)|^2\, dx\le CR^{N-2-\frac{4\nu}{p-1}}.
  \]
  The proof of Lemma~\ref{Lem:NRG1} is complete with the Sobolev inequality.
    \end{proof}
  \begin{Lemma}\label{Lem:Reg1}
    Assume the same conditions as in Theorem~{\rm\ref{Thm:Thm1}}, \eqref{eq:abcond}, and that $1<p<p_{JL}$. Then
    \begin{equation}\label{eq:Reg1}
      \sup_{\substack{\kappa\in(0,\kappa^*),\\x\in\Omega}}|x|^\frac{2}{p-1}u^\kappa(x)<\infty.
    \end{equation}
  \end{Lemma}
  \begin{proof}
  By $u^\kappa\le u^{\kappa^*}_{j_*}+w^\kappa_{j_*}$ and $u^{\kappa^*}_{j_*}\in C_{\alpha,\beta}$, it suffices to prove that 
  \[
      \sup_{\substack{\kappa\in(0,\kappa^*),\\x\in\Omega}}|x|^\frac{2}{p-1}w^\kappa(x)<\infty.
    \]
    It follows from \eqref{eq:NRG1} with $\nu=1$ and $q=1$ that
    \begin{equation}\label{eq:Reg1-0}
      \|w^\kappa\|_{L^2(\Omega_R)}\le CR^{\frac{N}{2}-\frac{2}{p-1}}.
    \end{equation}
    Furthermore, since $1<p<p_{JL}$, we find $\nu\ge 1$ and $q>1$ satisfying \eqref{eq:qcond} and
    \begin{equation}\label{eq:nucond2}
      \frac{\nu^2}{2\nu-1}>p,\quad \frac{2q}{p-1}\nu>\frac{N}{2}.
    \end{equation}
    Indeed, the case $N=2$ is clear. In the case $N=3$, we find $\nu\ge 1$ such that
    \[
    \frac{\nu^2}{2\nu-1}>p,\quad \frac{2N}{(N-2)(p-1)}\nu>\frac{N}{2}
    \]
    (see e.g., \cite{IOS01}*{Lemma 5.4}) and set $q=N/(N-2)$.
    It follows from \eqref{eq:NRG1} that
    \begin{equation}\label{eq:Reg1-1}
    \left(\int_{\widetilde\Omega_R} ((w^{\kappa})^{p-1})^{\frac{2q}{p-1}\nu}\, dx\right)^{\frac{p-1}{2q\nu}}
    \le CR^{\frac{p-1}{2\nu}\left(\frac{N}{q}-\frac{4\nu}{p-1}\right)}=CR^{\frac{N(p-1)}{2q\nu}-2}.
    \end{equation}
    On the other hand, it follows from 
$u^\kappa_{j_*}<u^{\kappa^*}_{j_*}$ and $u^{\kappa^*}_{j_*}\in C_{\alpha,\beta}$ that
    \begin{equation}\label{eq:Reg1-2}
    \left(\int_{\widetilde\Omega_R} ((u^{\kappa}_{j_*})^{p-1})^{\frac{2q}{p-1}\nu}\, dx\right)^{\frac{p-1}{2q\nu}}\le CR^{\frac{N(p-1)}{2q\nu}}\left(R^\alpha(1+R^2)^\frac{\beta-\alpha}{2}\right)^{p-1}\le CR^{\frac{N(p-1)}{2q\nu}-2}.
    \end{equation}
    Since
    \begin{equation}\label{eq:wprob'}
    0\le -\Delta w^\kappa\le p(w^\kappa+u^\kappa_{j_*})^{p-1}(w^\kappa+v^\kappa)\quad\textrm{in}\quad\Omega
    \end{equation}
    by \eqref{eq:wprob}, we deduce from \eqref{eq:a*b*}, \eqref{eq:Reg1-0}--\eqref{eq:Reg1-2}, and elliptic regularity theorems (see e.g., \cite{GT}*{Theorem 8.17}) that
    \begin{align*}
    w^\kappa(x)&\le C\left(|x|^{-\frac{N}{2}}\|w^\kappa\|_{L^2(\Omega_{|x|})}+|x|^{2-\frac{N(p-1)}{2\nu}}\|(w^\kappa+u^\kappa_{j_*})^{p-1}v^\kappa\|_{L^{\frac{2q}{p-1}\nu}(\Omega_{|x|})}\right)\\
    &\le C\left(|x|^{-\frac{2}{p-1}}+\sup_{\Omega_{|x|}}v^{\kappa^*}\right)\le C\left(|x|^{-\frac{2}{p-1}}+|x|^{\alpha_*}(1+|x|^2)^\frac{\beta_*-\alpha_*}{2}\right)\le C|x|^{-\frac{2}{p-1}},
    \end{align*}
    and the proof of Lemma~\ref{Lem:Reg1} is complete.
  \end{proof}
  \begin{Lemma}\label{Lem:NRG2}
    Assume the same conditions as in Theorem~{\rm\ref{Thm:Thm1}}. Let $\tau\in\R$ be such that
    \begin{equation}\label{eq:taucond}
    \begin{gathered}
      -\frac{N-2}{2}-\alpha<\tau<-\frac{N-2}{2}-\beta,\\
      \frac{1}{p}+\left(\left(\frac{N-2}{2}\right)^2+\Lambda\right)^{-1}\tau^2<1.
    \end{gathered}
    \end{equation}
    Then
    \begin{equation}\label{eq:NRG2}
      \sup_{\kappa\in(0,\kappa^*)}\int_{\Omega}|x|^{2\tau-2} (w^\kappa)^2\, dx<\infty.
    \end{equation}
  \end{Lemma}
  \begin{proof}
    Let $0<r<1$, $R>4$, and $\psi\in C^\infty(\R^N)$ be such that
    \begin{equation}\label{eq:psicut}
    \begin{gathered}
    \operatorname{\supp}\psi\subset\left\{x\in\R^N:\frac{r}{2}<|x|<2R\right\}, \quad 0\le\psi\le 1\quad\textrm{in}\quad\R^N,\\
    \psi=1\quad\textrm{on}\quad \{x\in\R^N:r<|x|<R\},\\
    |\nabla\psi|\le 4r^{-1}\quad\textrm{on}\quad \Omega_r,\quad |\nabla\psi|\le 2R^{-1}\quad\textrm{on}\quad \Omega_R.
    \end{gathered}
    \end{equation}
    We observe that
    \begin{equation}\label{eq:NRG2-2}
      \begin{aligned}
      &\int_\Omega|\nabla(\psi|x|^\tau w^\kappa)|^2\, dx\\
      &=\int_\Omega\left(\nabla w^\kappa\cdot\nabla(\psi^2|x|^{2\tau}w^\kappa)+(\tau^2|x|^{2\tau-2}+2\psi\nabla\psi\cdot \tau|x|^{2\tau-2}x+|\nabla\psi|^2|x|^{2\tau})(w^\kappa)^2\right)\, dx\\
&\le\int_\Omega\left(\psi^2|x|^{2\tau}\left((u^\kappa)^p-(u^\kappa_{j_*})^p\right)(w^\kappa)^2+\tau^2|x|^{2\tau-2}(w^\kappa)^2\right)\, dx+C\left(\rho_1[w^\kappa]+\rho_2[w^\kappa]\right),
      \end{aligned}
    \end{equation}
    where
    \[
      \rho_1[w]:=\int_\Omega \psi|\nabla\psi||x|^{2\tau-1}w^2\,dx,\quad \rho_2[w]:=\int_\Omega|\nabla\psi|^2|x|^{2\tau}w^2\, dx,
    \]
  for $w\in C(\overline{\Omega}\setminus\{0\})$.

    Let $\varepsilon,\delta\in(0,1)$. By the same argument as in \eqref{eq:NRG1-2} and \eqref{eq:NRG1-3},
    \begin{gather}
    \begin{multlined}
    \int_\Omega\psi^2|x|^{2\tau}((u^\kappa)^p-(u^\kappa_{j_*})^p)w^\kappa \, dx\\
    \le \int_\Omega \left((u^\kappa)^{p-1}\psi^2|x|^{2\tau}((1+\varepsilon)^2(w^\kappa)^2+C(v^\kappa)^2)\right.\\
    \qquad\left.+C(u^\kappa_{j_*})^{p-1+2\delta}\psi^2|x|^{2\tau}(w^\kappa+v^\kappa)^{2(1-\delta)}\right)\, dx,
      \end{multlined}\label{eq:NRG2-3}\\
      \int_\Omega (u^\kappa)^{p-1}\psi^2|x|^{2\tau}(w^\kappa)^2\, dx\le\frac{1}{p}\int_\Omega|\nabla(\psi|x|^\tau w^\kappa)|^2\, dx.\label{eq:NRG2-4}
    \end{gather}
    Furthermore, it follows from Lemma~\ref{Lem:Reg1} that
    \begin{equation}\label{eq:NRG2-5}
      \int_\Omega (u^\kappa)^{p-1}\psi^2|x|^{2\tau}(v^\kappa)^2\, dx\le C\int_\Omega |x|^{2\tau-2}(v^{\kappa^*})^2\, dx\le C\int_\Omega|x|^{2\tau-2}U_{\alpha_*,\beta_*}^2\,dx\le C.
    \end{equation}
    Here, we used the fact that \eqref{eq:taucond} implies that
    \begin{equation}\label{eq:abtau}
    2\tau+N+2\alpha-2>0,\quad 2\tau+N+2\beta-2<0.
    \end{equation}
    Furthermore, it follows from \eqref{eq:abtau} that
    \begin{equation}\label{eq:NRG2-6}
    \int_\Omega|x|^{2\tau-2}(u^{\kappa^*}_{j_*})^2\,dx\le C\int_\Omega U_{2\tau+2\alpha-2,2\tau+2\beta-2}\,dx\le C.
    \end{equation}
    Let $\delta>0$ be so small that
    \[
    (1-\delta)^{-1}(p-1)\alpha>-2-2\delta(1-\delta)^{-1},\quad (1-\delta)^{-1}(p-1)\beta<-2-2\delta(1-\delta)^{-1},
    \]
    which implies that
    \[
    (u^{\kappa^*}_{j_*})^{(1-\delta)^{-1}(p-1)}\le C|x|^{-2-2\delta(1-\delta)^{-1}}.
    \]
    Combining \eqref{eq:NRG2-6} with the H\"{o}lder inequality, the Hardy inequality \eqref{eq:Hardy}, and \eqref{eq:abtau}, we see that
    \begin{equation}\label{eq:NRG2-7}
      \begin{aligned}
        &\int_\Omega (u^\kappa_{j_*})^{p-1+2\delta}\psi^2|x|^{2\tau}(w^\kappa+v^\kappa)^{2(1-\delta)}\, dx\\
        &\le C\int_\Omega (u^{\kappa^*}_{j_*})^{p-1+2\delta}\psi^2|x|^{2\tau-2}(w^\kappa+(v^{\kappa^*})^{2(1-\delta)})^{2(1-\delta)}\, dx\\
 &\begin{multlined}\le C\left(\int_\Omega |x|^{2\tau}(u^{\kappa^*}_{j_*})^2\,dx\right)^{\delta}\\
 \quad\left(\int_\Omega (u^{\kappa^*}_{j_*})^{(1-\delta)^{-1}(p-1)}\psi^2|x|^{2\tau+2\delta(1-\delta)^{-1}}\left((w^\kappa)^2+(v^{\kappa^*})^2\right)\,dx\right)^{1-\delta}
  \end{multlined}\\
 &\le C\left(\int_\Omega |\nabla(\psi|x|^\tau(w^\kappa))|^2\,dx\right)^{(1-\delta)^{-1}}+C\left(\int_\Omega |x|^{2\tau-2}U_{\alpha_*,\beta_*}^2\,dx\right)^{(1-\delta)^{-1}}\\
 &\le \varepsilon'\int_\Omega |\nabla(\psi|x|^\tau(w^\kappa))|^2\,dx+C
      \end{aligned}
    \end{equation}
    for any $\varepsilon'>0$. On the other hand, by the Hardy inequality on infinite cones \eqref{eq:Hardy}, we obtain
    \begin{equation}\label{eq:NRG2-8}
      \int_\Omega\psi^2|x|^{2\tau-2}(w^\kappa)^2\, dx\le \left(\left(\frac{N-2}{2}\right)^2+\Lambda\right)^{-1}\int_\Omega|\nabla(\psi|x|^\tau w^\kappa)|^2\, dx.
    \end{equation}
    Combining \eqref{eq:NRG2-2}--\eqref{eq:NRG2-5}, \eqref{eq:NRG2-8}, and \eqref{eq:NRG2-9}, we deduce that
    \begin{multline*}
      \int_\Omega|\nabla(\psi|x|^\tau w^\kappa)|^2\, dx
      \le \left(\frac{(1+\varepsilon)^2}{p}+\left(\left(\frac{N-2}{2}\right)^2+\Lambda\right)^{-1}\tau^2+C_\varepsilon\varepsilon'\right)\int_\Omega|\nabla(\psi|x|^\tau w^\kappa)|^2\, dx\\
      \qquad\qquad+C_{\varepsilon,\varepsilon'}(1+\rho_1[w^{\kappa}]+\rho_2[w^{\kappa}])
    \end{multline*}
    for any $\varepsilon,\varepsilon'>0$. By \eqref{eq:taucond} and a suitable choice of $\varepsilon$ and $\varepsilon'$, we obtain
    \begin{equation}\label{eq:NRG2-9}
    \begin{multlined}
      \int_\Omega |x|^{2\tau-2}(w^\kappa)^2\, dx\le C\int_\Omega|\nabla(\psi|x|^\tau w^\kappa)|^2\, dx\le C(1+\rho_1[w^{\kappa}]+\rho_2[w^{\kappa}]).
    \end{multlined}
    \end{equation}
    Here, for any $w\in C_{\alpha_*,\beta_*}$, it follows from \eqref{eq:psicut} that
    \[
      \rho_1[w]+\rho_2[w]\le C\left(r^{2\tau-2}\int_{\Omega_r}w^2\, dx+R^{2\tau-2}\int_{\Omega_R}w^2\, dx\right)\le C\left(r^{2\tau-2}r^{N+2\alpha_*}+R^{2\tau-2}R^{N+2\beta_*}\right).
    \]
    This together with \eqref{eq:abtau} implies that
    \[
    \rho_1[w]+\rho_2[w]\to 0\quad \textrm{as}\quad r\to+0,\quad R\to\infty.
    \]
    Consequently, letting $r\to +0$ and $R\to \infty$ on \eqref{eq:NRG2-9}, we complete the proof of Lemma~\ref{Lem:NRG2}.
  \end{proof}
  \begin{Lemma}\label{Lem:Reg2}
    Assume the same conditions as in Theorem~{\rm\ref{Thm:Thm1}}, \eqref{eq:abcond}, and $1<p<p_{JL}$. Assume additionally that
    \begin{equation}\label{eq:taucond'}
      \frac{1}{p}+\left(\left(\frac{N-2}{2}\right)^2+\Lambda\right)^{-1}\left(\frac{2}{p-1}-\frac{N-2}{2}\right)^2<1.
    \end{equation}
    Then
    \begin{equation}\label{eq:Reg2}
      \sup_{\kappa\in(0,\kappa^*)}\|w^\kappa\|_{C_{\alpha_*,\beta_*}}<\infty.
    \end{equation}
  \end{Lemma}
  \begin{proof}
    It follows from \eqref{eq:abcond} that
    \[
    -\frac{N-2}{2}-\alpha<\frac{2}{p-1}-\frac{N-2}{2}<-\frac{N-2}{2}-\beta.
    \]
    We find $\tau_1$ and $\tau_2$ satisfying \eqref{eq:qcond}
    \begin{equation}\label{eq:tau12choice}
      \begin{gathered}
        -\frac{N-2}{2}-\alpha<\tau_1<\frac{2}{p-1}-\frac{N-2}{2}<\tau_2<-\frac{N-2}{2}-\beta,\\
        \frac{1}{p}+\left(\left(\frac{N-2}{2}\right)^2+\Lambda\right)^{-1}\tau_i^2<1\quad \textrm{for}\quad i\in\{1,2\}.
      \end{gathered}
    \end{equation}
    Let $i\in\{1,2\}$. It follows from \eqref{eq:NRG2} that
    \[
    \|w^\kappa\|_{L^2(\Omega_R)}^2\le R^{2-\tau_i}\int_\Omega |x|^{2\tau_i-2}(w^\kappa)^2\,dx\le CR^{2-2\tau_i}
    \]
    for $R>0$. This together with \eqref{eq:Reg1}, \eqref{eq:wprob'}, and elliptic regularity theorems implies that
    \begin{align*}
      w^\kappa(x)&\le C\left(|x|^{-\frac{N}{2}}\|w^\kappa\|_{L^2(\Omega_{|x|})}+|x|^2\sup_{\Omega_{|x|}}(w^\kappa+u^\kappa_{j_*})^{p-1}v^\kappa\right)\\
      &\le C\left(|x|^{-\frac{N}{2}}|x|^{1-\tau_i}+|x|^{\alpha_*}(1+|x|^2)^\frac{\beta_*-\alpha_*}{2}\right).
    \end{align*}
    Consequently, we obtain
    \[
      w^\kappa(x)\le C|x|^{\widetilde{\alpha}}(1+|x|^2)^\frac{\widetilde{\beta}-\widetilde{\alpha}}{2},
    \]
    where
    \[
    \widetilde{\alpha}:=\min\left\{-\frac{N-2}{2}-\tau_1,\alpha_*\right\},\quad \widetilde{\beta}:=\max\left\{-\frac{N-2}{2}-\tau_2,\beta_*\right\}.
    \]
    Furthermore, it follows from \eqref{eq:abcond}, \eqref{eq:a*b*}, and \eqref{eq:tau12choice} that
    \[
    \widetilde{\alpha}>-\frac{2}{p-1}>\widetilde{\beta}.
    \]
    Applying a similar argument to the proof of Lemma~\ref{Lem:wk}, we deduce that $\|w^\kappa\|_{C_{\alpha_*,\beta_*}}\le C$, and complete the proof of Lemma~\ref{Lem:Reg2}.
  \end{proof}
  Now we are ready to prove Proposition~\ref{Prop:Sec4}.
  \begin{proof}[Proof of Proposition {\rm\ref{Prop:Sec4}}]
    The condition \eqref{eq:taucond'} is equivalent to $H(p-1)<0$ by direct calculation. Indeed, we see that \eqref{eq:taucond'} is equivalent to $H<0$, where
    \begin{align*}
      H&:=-4\left(\left(\frac{N-2}{2}\right)^2+\Lambda\right)(p-1)^3+p(4-(N-2)(p-1))^2\\
      &\begin{multlined}=-((N-2)^2+4\Lambda)(p-1)^3+(N-2)^2(p-1)^3+((N-2)^2-8(N-2))(p-1)^2\\
      +(-8(N-2)+16)(p-1)+16\end{multlined}\\
      &=-4\Lambda(p-1)^3+(N-2)(N-10)(p-1)^2-8(N-4)(p-1)+16\\
      &=H(p-1).
    \end{align*}
    It follows from Lemma~\ref{Lem:wk} and \eqref{eq:Reg2} that the limit
    \[
    w^*(x):=\lim_{\kappa\to\kappa^*-0}w^\kappa(x)
    \]
  exists for each $x\in\Omega$. Furthermore, \eqref{eq:Reg2} and elliptic regularity theorems imply that $\{w^\kappa\}_{\kappa\in(0,\kappa^*)}$ is bounded in $C_{\alpha_*,\beta_*}$, and that $\{\nabla^2 w^\kappa\}_{\kappa\in(0,\kappa^*)}$ is bounded in $C^{0,\tau}(\Omega')$ for each compact subset $\Omega'$ of $\Omega$. Thus, $w^\kappa$ belongs to $C^2(\Omega)\cap C_{\alpha_*,\beta_*}$ and satisfies
  \[
    -\Delta w^*=(w^*+u^{\kappa^*}_{j_*+1})^p-(u^{\kappa^*}_{j_*})^p\quad\textrm{in}\quad\Omega,\quad w^*=0\quad\textrm{on}\quad\partial\Omega\setminus\{0\}.
   \]
   Consequently, $u^*:=w^*+u^{\kappa^*}_{j_*+1}\in C^2(\Omega)\cap C_{\alpha,\beta}$ is a solution to problem~\eqref{L}. The proof of Proposition~\ref{Prop:Sec4} is complete.
  \end{proof}
  \section{Proof of Theorem \ref{Thm:Thm2}}
  In this section, we prove Theorem \ref{Thm:Thm2}. 
We first note that $u\in C_{\alpha,\beta}$ is a solution to problem~\eqref{L} if and only if
\[
\Phi(u,\kappa):=u-\kappa u^1_0-G[u_+^p]=0,
\]
  where $G:C_{p\alpha,p\beta}\to C_{p\alpha+2,p\beta+2}$ is as in Section~2, and $u_+:=\max\{u,0\}$. Indeed, the ``only if'' part is obvious. The ``if'' part holds since $u=\kappa u^1_0+G[u^p]$ and 
Proposition~\ref{Prop:G}~(i) imply that $u\in C^{0,\tau}_{\loc}(\Omega)$. It is also clear that $\Phi\in C^1(C_{\alpha,\beta}\times(0,\infty),C_{\alpha,\beta})$ and
  \begin{equation}\label{eq:Phidif}
  \Phi_u(u,\kappa)h=h-G[pu_+^{p-1}h],\quad \Phi_\kappa(u,\kappa)=-u^1_0,
  \end{equation}
  for any $u,h\in C_{\alpha,\beta}$, $\kappa>0$.
  \begin{Lemma}\label{Lem:lam*}
    Assume the same conditions as in Theorem~{\rm\ref{Thm:Thm2}}. Let $\kappa\in\mathcal{K}$. If $\lambda^\kappa>1$, then $\kappa<\kappa^*$. Furthermore, $\lambda^{\kappa^*}=1$.
  \end{Lemma}
  \begin{proof}
    We first prove that $\lambda^\kappa>1\implies \kappa<\kappa^*$. It suffices to prove that $\Phi_u(u^\kappa,\kappa): C_{\alpha,\beta}\to C_{\alpha,\beta}$ is an isomorphism of Banach spaces. Indeed, this together with the implicit function theorem implies that $\Phi(U,\kappa')=0$ for some $U\in C_{\alpha,\beta}$ and $\kappa'>\kappa$, and $U$ is a solution to problem \eqref{L} with $\kappa$ replaced by $\kappa'$. Thus $\kappa<\kappa'\le\kappa^*$.

    We first prove that
    \begin{equation}\label{eq:Phidifker1}
      \operatorname{Ker}\Phi_u(u^\kappa,\kappa)=\{0\}.
    \end{equation}
    Let $\phi\in\operatorname{Ker}\Phi_u(u^\kappa,\kappa)$. It follows from \eqref{eq:Phidif} that
    \[
    -\Delta\phi=p(u^\kappa)^{p-1}\phi\quad\textrm{in}\quad\Omega,\quad \phi=0\quad\textrm{on}\quad\partial\Omega\setminus\{0\}.
    \]
    By a similar argument to the proof of Lemma~\ref{Lem:vk}, we observe that $\phi\in C_{\alpha_j,\beta_j}$ for all $j\in\{1,2,\ldots\}$, and thus $\phi\in C_{\alpha_*,\beta_*}$. This together with Proposition~\ref{Prop:G}~(iii) and \eqref{eq:D12check} implies that $\phi\in\mathcal{D}^{1,2}_0(\Omega)$. But since $\lambda^\kappa>1$, this implies that $\phi\equiv 0$. Thus \eqref{eq:Phidifker1} holds.

    It remains to prove that
    \begin{equation}\label{eq:Phidifim1}
      \operatorname{Im}\Phi_u(u^\kappa,\kappa)=C_{\alpha,\beta}.
    \end{equation}
    Since the operator $G$ is compact from $C_{p\alpha,p\beta}$ to $C_{\alpha,\beta}$ by \eqref{eq:paa2} and Proposition~\ref{Prop:G}~(ii), the operator $h\mapsto G[(u^\kappa)^{p-1}h]$ is compact on $C_{\alpha,\beta}$. This together with \eqref{eq:Phidif} and the Fredholm alternative theorem implies that \eqref{eq:Phidifker1} implies \eqref{eq:Phidifim1}. Thus $\lambda^\kappa>1\implies \kappa<\kappa^*$ holds.

    We next prove that $\lambda^{\kappa^*}=1$. It suffices to prove that $\lambda^{\kappa^*}\ge 1$. It follows from the proof of Proposition~\ref{Prop:Sec4} that
    \[
    u^{\kappa^*}\le u^*=u^{\kappa^*}_{j_*+1}+w^*=\lim_{\kappa\to\kappa^*-0}(u^\kappa_{j_*+1}+w^\kappa)=\lim_{\kappa\to\kappa^*-0}u^\kappa.
    \]
    This together with \eqref{eq:stab} implies that
    \[
    \int_\Omega p(u^{\kappa^*})^{p-1}\phi^2\, dx\le\liminf_{\kappa\to\kappa^*-0}\int_\Omega p(u^{\kappa})^{p-1}\phi^2\, dx\le\int_\Omega|\nabla\phi|^2\, dx
    \]
    for any $\phi\in\mathcal{D}^{1,2}_0(\Omega)$. This derives $\lambda^{\kappa^*}\ge 1$, and the proof of Lemma~\ref{Lem:lam*} is complete.
  \end{proof}
  \begin{Lemma}\label{Lem:k*imker}
    Assume the same conditions as in Theorem~{\rm\ref{Thm:Thm1}}. Then
    \begin{align}
      \operatorname{Ker}\Phi_u(u^{\kappa^*},\kappa^*)&=\operatorname{span}\{\phi^{\kappa^*}\},\label{eq:Phidifker2}\\
      \operatorname{Im}\Phi_u(u^{\kappa^*},\kappa^*)&=\left\{h\in C_{\alpha,\beta}:\ \int_\Omega p(u^{\kappa^*})^{p-1}\phi^{\kappa^*}h\, dx=0\right\}.\label{eq:Phidifim2}
    \end{align}
  \end{Lemma}
  \begin{proof}
    We first prove \eqref{eq:Phidifker2}. By the same argument as in the proof of \eqref{eq:Phidifker1}, we observe that
    \[
    \operatorname{Ker}\Phi_u(u^{\kappa^*},\kappa^*)\subset\left\{\phi\in\mathcal{D}^{1,2}_0(\Omega):\ -\Delta\phi=p(u^{\kappa^*})^{p-1}\phi\quad\textrm{in}\quad\Omega\right\}.
    \]
    This together with \eqref{eq:EVP} and Lemma~\ref{Lem:lam*} implies that $ \operatorname{Ker}\Phi_u(u^{\kappa^*},\kappa^*)\subset\operatorname{span}\{\phi^{\kappa^*}\}$. On the other hand, since $\lambda^{\kappa^*}=1$, \eqref{eq:phiG} implies that $\phi^{\kappa^*}\in \operatorname{Ker}\Phi_u(u^{\kappa^*},\kappa^*)$. Hence \eqref{eq:Phidifker2} holds.

    We next prove \eqref{eq:Phidifim2}. Since
    \[
    (u^{\kappa^*})^{p-1}\phi^{\kappa^*}\in C_{(p-1)\alpha+\alpha_*,(p-1)\beta+\beta_*}
    \]
    and
    \[
    (p-1)\alpha+\alpha_*+\alpha>-N,\quad (p-1)\beta+\beta_*+\beta<-N
    \]
    by \eqref{eq:a*b*2}, the closed subspace
    \begin{equation}\label{eq:Z}
      Z:=\left\{h\in C_{\alpha,\beta}: \int_\Omega p(u^{\kappa^*})^{p-1}\phi^{\kappa^*}h\, dx=0\right\}
    \end{equation}
    of $C_{\alpha,\beta}$ is well-defined and satisfies $\operatorname{codim}Z=1$.
    
    It follows from \eqref{eq:Phidifker2} and the Fredholm alternative theorem that
    \begin{equation}\label{eq:Phidifim2-1}
      \operatorname{codim}\operatorname{Im}\Phi_u(u^{\kappa^*},\kappa^*)=\operatorname{dim}\operatorname{Ker}\Phi_u(u^{\kappa^*},\kappa^*)=1.
    \end{equation}
    We prove that
    \begin{equation}\label{eq:Phidifim2-2}
      \operatorname{Im}\Phi_u(u^{\kappa^*},\kappa^*)\subset Z.
    \end{equation}
    Let $h\in \operatorname{Im}\Phi_u(u^{\kappa^*},\kappa^*)$ and $f\in C_{\alpha,\beta}$ be such that $\Phi_u(u^{\kappa^*},\kappa^*)f=h$. It follows from Proposition~\ref{Prop:G}~(iii) and \eqref{eq:Phidif} that
    \begin{equation}\label{eq:Phidifim2-3}
      -\Delta(f-h)=p(u^{\kappa^*})^{p-1}f\jn\Omega,\quad f-h\in D^{1,2}_0(\Omega).
    \end{equation}
    Let $\psi\in C^\infty(\R^N)$ be as in \eqref{eq:psicut}. It follows from \eqref{eq:EVP} and \eqref{eq:Phidifim2-3} that
    \begin{multline*}
      \int_\Omega p(u^{\kappa^*})^{p-1}\phi^{\kappa^*}\psi(f-h)\, dx=\int_\Omega \nabla\phi^{\kappa^*}\cdot\nabla(\psi(f-h))\, dx\\
      =\int_\Omega \nabla(\psi\phi^{\kappa^*})\cdot\nabla(f-h)\, dx+\int_\Omega((\nabla\phi^{\kappa^*}\cdot\nabla\psi)(f-h)-\phi^{\kappa^*}\nabla\psi\cdot\nabla(f-h))\, dx\\
      =\int_\Omega p(u^{\kappa^*})^{p-1}\psi\phi^{\kappa^*}f\, dx+\int_\Omega (2\nabla\phi^{\kappa^*}\cdot\nabla\psi+\phi^{\kappa^*}\Delta\psi)(f-h)\, dx,
    \end{multline*}
    and thus
    \begin{equation}\label{eq:Phidifim2-4}
      \int_\Omega p(u^{\kappa^*})^{p-1}\phi^{\kappa^*}\psi h\, dx=\int_\Omega (2\nabla\phi^{\kappa^*}\cdot\nabla\psi+\phi^{\kappa^*}\Delta\psi)(f-h)\, dx.
    \end{equation}
    Here, it follows from $f-h\in C_{\alpha,\beta}$, Lemma~\ref{Lem:phiimp}, and \eqref{eq:psicut} that
    \begin{align*}
      &\left|\int_\Omega (2\nabla\phi^{\kappa^*}\cdot\nabla\psi+\phi^{\kappa^*}\Delta\psi)(f-h)\, dx\right|\\
      &\le C\left(r^\frac{N}{2}\|\nabla\phi^{\kappa^*}\|_{L^2(\Omega_\frac{r}{2})}\sup_{\Omega_\frac{r}{2}}|\nabla\psi|+r^N\sup_{\Omega_\frac{r}{2}}(\phi^{\kappa^*}|\Delta\psi|)\right)r^{\alpha}(1+r^2)^\frac{\beta-\alpha}{2}\\
      &\qquad+C\left(R^\frac{N}{2}\|\nabla\phi^{\kappa^*}\|_{L^2(\Omega_{R})}\sup_{\Omega_{R}}|\nabla\psi|+R^N\sup_{\Omega_R}(\phi^{\kappa^*}|\Delta\psi|)\right)R^{\alpha}(1+R^2)^\frac{\beta-\alpha}{2}\\
      &\le C\left(r^\frac{N-2}{2}\cdot r^{\alpha_*+\frac{N-2}{2}}(1+r^2)^{\frac{\beta_*-\alpha_*}{2}}+r^{N-2}\cdot r^{\alpha_*}(1+r^2)^\frac{\beta_*-\alpha_*}{2}\right)r^{\alpha}(1+r^2)^\frac{\beta-\alpha}{2}\\
      &\qquad +C\left(R^\frac{N-2}{2}\cdot R^{\alpha_*+\frac{N-2}{2}}(1+R^2)^{\frac{\beta_*-\alpha_*}{2}}+R^{N-2}\cdot R^{\alpha_*}(1+R^2)^\frac{\beta_*-\alpha_*}{2}\right)R^{\alpha}(1+R^2)^\frac{\beta-\alpha}{2}\\
      &\le C\left(r^{N-2+\alpha_*+\alpha}+R^{N-2+\beta_*+\beta}\right).
    \end{align*}
    This together with
    \[
    N-2+\alpha_*+\alpha>0,\quad N-2+\beta_*+\beta<0
    \]
    by \eqref{eq:a*b*3} implies that
    \[
    \lim_{r\to+0,R\to\infty}\left|\int_\Omega (2\nabla\phi^{\kappa^*}\cdot\nabla\psi+\phi^{\kappa^*}\Delta\psi)(h-f)\, dx\right|=0.
    \]
    Consequently, letting $r\to+0$ and $R\to\infty$ in \eqref{eq:Phidifim2-4}, we deduce that
    \[
     \int_\Omega p(u^{\kappa^*})^{p-1}\phi^{\kappa^*}\psi h\, dx=0,
    \]
    which implies that $h\in Z$.

    Since $\operatorname{codim}Z=1$, \eqref{eq:Phidifim2-1} and \eqref{eq:Phidifim2-2} imply \eqref{eq:Phidifim2}. The proof of Lemma~\ref{Lem:k*imker} is complete.    
  \end{proof}
  
  Now we are ready to prove Theorem~\ref{Thm:Thm2}.
  \begin{proof}[Proof of Theorem {\rm\ref{Thm:Thm2}}]
    By Theorem~\ref{Thm:Thm1} and Proposition~\ref{Prop:Sec4}, it remains to prove the uniqueness of solution $u\in C^2(\Omega)\cap C_{\alpha,\beta}$ to problem~\eqref{L} with $\kappa=\kappa^*$, and the multiplicity assertion (iii).

    We first prove the uniqueness of solution $u\in C^2(\Omega)\cap C_{\alpha,\beta}$ to problem~\eqref{L} with $\kappa=\kappa^*$. Let $\widetilde{u}\in C^2(\Omega)\cap C_{\alpha,\beta}$ be a solution to problem~\eqref{L} with $\kappa=\kappa^*$. It follows from the same argument as in Lemma~\ref{Lem:wk} that
    \[
    \widetilde{w}:=\widetilde{u}-u^{\kappa^*}_{j_*+1}\in\mathcal{D}^{1,2}_0(\Omega),\quad -\Delta\widetilde{w}=(\widetilde{u})^p-(u^{\kappa^*}_{j_*})^p.
    \]
    We see that
    \[
   \widetilde{u}-u^{\kappa^*}_{j_*}=\widetilde{w}-w^{\kappa_*}_{j_*}\in\D^{1,2}_0(\Omega),\quad -\Delta(\widetilde{u}-u^{\kappa_*}_{j_*})=(\widetilde{u})^p-(u^{\kappa^*})^p\quad\textrm{in}\quad\Omega
    \]
    in the weak sense. This together with Lemma~\ref{Lem:lam*} implies that
    \[
    \int_\Omega p(u^{\kappa^*})^{p-1}\phi^{\kappa^*}(\widetilde{u}-u^{\kappa^*})\, dx=\int_\Omega\nabla\phi^{\kappa^*}\cdot\nabla(\widetilde{u}-u^{\kappa^*})\, dx=\int_\Omega \left((\widetilde{u})^p-(u^{\kappa^*})^p\right)\phi^{\kappa^*}\, dx.
    \]
    Since $t^p-s^p>ps^{p-1}(t-s)$ for any $t>s>0$, we consequently obtain $\widetilde{u}=u^{\kappa^*}$ in $\Omega$.
    
    We finally prove assertion (iii). Let $Z\subset C_{\alpha,\beta}$ be as in \eqref{eq:Z}. It follows from \eqref{eq:Phidif} that
    \[
    \Phi_\kappa(u^{\kappa^*},\kappa^*)<0\quad\textrm{in}\quad\Omega.
    \]
    We see that
    \[
    \int_\Omega p(u^{\kappa^*})^{p-1}\phi^{\kappa^*}\Phi_\kappa(u^{\kappa^*},\kappa^*)\, dx<0,
    \]
    and $\Phi_\kappa(u^{\kappa^*},\kappa^*)\notin Z$ by \eqref{eq:Phidifim2}. This together with Lemma~\ref{Lem:k*imker} enables us to apply the following bifurcation theorem.
  \begin{Proposition}\cite{CR1}*{Theorem 3.2}
Let $X$, $Y$ be two Banach spaces, $U\subset X\times\R$ be an open set, and $F=F(u,\kappa):U\to Y$ is a $C^1$ map. Let $(u^*,\kappa^*)\in U$ be such that
\begin{gather*}
F(u^*,\kappa^*)=0,\quad \dim\left(\Ker (F_u(u^*,\kappa^*))\right)=\operatorname{codim} (\Im F_u(u^*,\kappa^*))=1,\\
F_\kappa(u^*,\kappa^*)\notin \Im (F_u(u^*,\kappa^*)).
\end{gather*}
Let $\phi\in X$ be such that $\Ker(F_u(u^*,\kappa^*))=\operatorname{span}\{\phi\}$, and $Z\subset X$ be a complementary subspace of $\Ker(F_u(u^*,\kappa^*))$, and assume that $F_\kappa(u^*,\kappa^*)\notin Z$. Then solutions $(u,\kappa)$ to equation $F(u,\kappa)=0$ near $(u^*,\kappa^*)$ form a $C^1$ curve
\[
(u(s),\kappa(s))=(u^*+s\phi+z(s),\kappa^*+\tau(s)),
\]
with $(z,\tau):(-\delta,\delta)\to Z\times\R$, $z(0)=0$,  $\tau(0)=0$, and $\tau'(0)=0$.
\end{Proposition}
  Using this theorem with $X=Y=C_{\alpha,\beta}$, $F=\Phi$, and $u^*=u^{\kappa^*}$, we obtain a $C^1$ curve $(u(s),\kappa(s))=(u^{\kappa^*}+s\phi^{\kappa^*}+z(s),\kappa^*+\tau(s))$ with $(z,\tau):(-\delta,\delta)\to Z\times\R$ with $\delta>0$ sufficiently small, such that $u(s)$ is a solution to problem~\eqref{L} with $\kappa=\kappa(s)$ for all $s\in(-\delta,\delta)$, and that
  \begin{equation}\label{eq:curve}
      z(0)=0,\quad \tau(0)=0,\quad \tau'(0)=0,
  \end{equation}
  Furthermore, it follows from \eqref{eq:Z} that
  \[
  \int_\Omega p(u^{\kappa^*})^{p-1}\phi^{\kappa^*}(u(s)-u^{\kappa^*}_{j_*})\, dx=s\int_\Omega p(u^{\kappa^*})^{p-1}(\phi^{\kappa^*})^2\, dx,
  \]
  for each $s\in(-\delta,\delta)$, and thus
  \begin{equation}\label{eq:distinct}
  s_1,s_2\in(-\delta,\delta),\ s_1\neq s_2\implies u(s_1)\not\equiv u(s_2)\quad\textrm{in}\quad\Omega
  \end{equation}
  for $s_1,s_2\in(-\delta,\delta)$.

  Since $u(s)$ is a solution to problem~\eqref{L}with $\kappa=\kappa(s)$, we see that $\kappa(s)\le\kappa^*$ for any $s\in(-\delta,\delta)$. This together with $\kappa(0)=\kappa^*$, \eqref{eq:distinct}, and the uniqueness of a solution $u\in C^2(\Omega)\cap C_{\alpha,\beta}$ to problem \eqref{L} with $\kappa=\kappa^*$ implies that
  \[
  s\in(-\delta,\delta),\ s\neq 0\implies \kappa(s)<\kappa^*.                                             
  \]
  By the intermediate value theorem, there exists $\kappa_*\in[0,\kappa^*)$ such that for any $\kappa\in(\kappa_*,\kappa^*)$, there exist $s_1\in(-\delta,0)$ and $s_2\in(0,\delta)$ such that $\kappa(s_1)=\kappa(s_2)=\kappa$. By \eqref{eq:curve} and \eqref{eq:distinct}, $u(s_1)$ and $u(s_2)$ are two distinct solutions to problem \eqref{L}. The proof of Theorem~\ref{Thm:Thm2} is complete.
  \end{proof}
\section{Generalization: H\'{e}non-type equations}
  As a generalization of Theorems~\ref{Thm:Thm1}~and~\ref{Thm:Thm2}, we mention results on a boundary value problem for a H\'{e}non-type equation on cone:
\begin{equation}\label{H}\tag{\mbox{$\textrm{H}_\kappa$}}
    \left\{\begin{aligned}
      -\Delta u&=K(x)u^p\quad& &\textrm{in}\quad& &\Omega,\\
      u&>0\quad& &\textrm{in}\quad& &\Omega,\\
      u&=\kappa\mu& &\textrm{on}\quad& &\partial\Omega\setminus\{0\},
    \end{aligned}\right.
  \end{equation}
   where $K$ is a positive locally H\"{o}lder continuous function. We assume that there exist constants $c_1,c_2>0$ and $a\in\R$ such that
  \[
  c_1|x|^{a}\le K(x)\le c_2|x|^a\quad\textrm{for all}\quad x\in\Omega.
  \]
  We define a solution to problem~\eqref{H} in the same way as in Definition~\ref{Def:sols} with \eqref{eq:abcond} replaced by
  \[
  -N+2-\gamma<\beta<-\frac{2+a}{p-1}<\alpha<\gamma.
  \]
  We define $p^*_{a,\gamma}>0$ by
  \[
  p^*_{a,\gamma}:=\frac{N+a+\gamma}{N-2+\gamma}\quad\textrm{if}\quad a\ge -2,\quad p^*_{a,\gamma}:=1+\frac{-(2+a)}{\gamma}\quad\textrm{if}\quad a<-2.
  \]
  We note that the condition $p>p^*_{a,\gamma}$ is a necessary condition for the existence of a solution to problem \eqref{H}. Indeed, due to Laptev \cite{L}, there exists no function $u\in C^2(\Omega)$ such that
  \[
  -\Delta u\ge c_1|x|^au^p\jn\Omega,\quad u>0\jn\Omega.
  \]
  (See also e.g., \cite{BE}.)
  
  Now we state our results on the problem~\eqref{H}.
  \begin{Theorem}\label{Thm:Thm1'}
  Let $p>p^*_{a,\gamma}$, and $\mu\in C(\partial\Omega\setminus\{0\})$ be a nonnegative non-zero function such that 
  \[
  \limsup_{|x|\to+0}|x|^{-\alpha}\mu(x)<\infty,\quad \limsup_{|x|\to\infty}|x|^{-\beta}\mu(x)<\infty.
  \]
Then there exists a constant $\kappa^*>0$ with the following properties.
    \begin{enumerate}[label={\rm(\roman*)}]
      \item If $0<\kappa<\kappa^*$, then problem~\eqref{H} possesses a minimal solution.
      \item If $\kappa>\kappa^*$, then problem~\eqref{H} possesses no solutions.
    \end{enumerate}
  \end{Theorem}
  \begin{Theorem}\label{Thm:Thm2'}
  Assume the same conditions as in Theorem~{\rm{\ref{Thm:Thm1}}}. Define
  \[
  H_a(q):=-4\Lambda q^3+(N-2)(N-10-4a)q^2-8(N-4-a)q+4(2+a)^2,
  \]
  and assume additionally that 
  \begin{equation}\label{eq:pjl}
  p<p_{JL},\quad H_a(p-1)<0.               
  \end{equation}
  Let $\kappa^*$ be as in Theorem~{\rm\ref{Thm:Thm1'}}. Then the following properties hold.
    \begin{enumerate}[label={\rm(\roman*)}]
      \item If $\kappa=\kappa^*$, then problem~\eqref{H} possesses a unique solution.
      \item there exists a constant $\kappa_*\in[0,\kappa^*)$ such that if $\kappa_*<\kappa<\kappa^*$, then problem \eqref{H} possesses at least two solutions.
    \end{enumerate}
  \end{Theorem}
  The range \eqref{eq:pjl} can be rewritten in one of the following forms:
  \[
  \textrm{(i) }(p_1,p_2)\cup(p_3,p_{JL});\quad\textrm{(ii) }(p_1,p_2);\quad\textrm{(iii) }(p_1,p_{JL});\quad\textrm{(iv) }\emptyset.
  \]
  Here, in the cases (i) and (ii), $p_1$, $p_2$, and $p_3$ are roots of cubic equation $H_a(p-1)=0$ and satsify
  \begin{gather*}
    p_{a,\gamma}^*<p_1<p_2\le p_3<p_{JL}\quad\textrm{in the case (i)},\\
    p_{a,\gamma}^*<p_1<p_2\le p_{JL}\le p_3\quad\textrm{in the case (ii)}.
  \end{gather*}
  In the case (iii), there exists only one root $p_1$ of the equation $H_a(p-1)=0$ in the range $(p_{a,\gamma}^*,p_{JL})$. In the case (iv), there exists no root of the equation $H_a(p-1)=0$ in the range $(p_{a,\gamma}^*,p_{JL})$. Note also that only the case (iii) can occur if $a=0$.

  The proofs of Theorems~\ref{Thm:Thm1'}~and~\ref{Thm:Thm2'} are similar to those of Theorems~\ref{Thm:Thm1}~and~\ref{Thm:Thm2}. We explain some parts of proofs for the convenience of the reader. As in Section~3, we obtain the existence of a solution to problem~\eqref{H} for small $\kappa$ by the inequality
  \[
  -\Delta V_{\alpha,\beta}\ge CU_{p\alpha+a,p\beta+a}\ge CKV^p_{\alpha,\beta}\jn\Omega.
  \]
  As in Section~4, we find $\alpha_*$ and $\beta_*$ satisfying \eqref{eq:a*b*}, \eqref{eq:a*b*3}, and
  \[
    p\alpha+a+\alpha_*>-N,\quad p\beta+a+\beta_*<-N,
  \]
  and show that the first eigenfunction $\phi^\kappa$ for the eigenvalue problem
  \[
  -\Delta\phi=p\lambda Ku^{p-1}\phi\jn\Omega,\quad \phi\in\D^{1,2}_0(\Omega)
  \]
  belongs to $\phi^\kappa\in C_{\alpha_*,\beta_*}$. Furthermore, for the approximate solutions $\{u^\kappa_j\}_{j\in\{-1,0,\ldots\}}$ defined by
  \[
  -\Delta u^\kappa_{j+1}=K(u^\kappa_j)^p\jn\Omega,\quad u^\kappa_{j+1}=\kappa\mu\on\partial\Omega\setminus\{0\},
  \]
  the function $w^\kappa:=u^\kappa-u^\kappa_{j_*+1}$ belongs to $C_{\alpha_*,\beta_*}$ for some large $j_*$. Moreover, the stability inequality
  \[
  \int_\Omega pK(u^\kappa)^{p-1}\phi^2\,dx\le\int_\Omega|\nabla\phi|^2\,dx\quad\textrm{for any}\quad\phi\in\D^{1,2}_0(\Omega)
  \]
  can be obtained. This yields Theorem~\ref{Thm:Thm1'}.
  
  Using the stability inequality as in Section~5, we derive $\sup\limits_{\kappa\in(0,\kappa^*)}|x|^\frac{2+a}{p-1}u^\kappa(x)<\infty$ if $p<p_{JL}$. Furthermore, we obtain the same estimate as in Lemma~\ref{Lem:NRG2}, and $\sup\limits_{\kappa\in(0,\kappa^*)}\|w^\kappa\|_{C_{\alpha_*,\beta_*}}<\infty$ if
  \[
  \frac{1}{p}+\left(\left(\frac{N-2}{2}\right)^2+\Lambda\right)^{-1}\left(\frac{2+a}{p-1}-\frac{N-2}{2}\right)^2<1,
  \]
  which is equivalent to $H_a(p-1)<0$. The rest of the proof of Theorem~\ref{Thm:Thm2'} is same as in Section~6.
  \begin{Remark}
 As a further generalization, our method can be applied to a boundary value problem of an H\'{e}non-type equation on a general unbounded domain
     \begin{equation}\label{eq:GBP}
    \left\{\begin{aligned}
      -\Delta u&=K(x)u^p\quad& &\textrm{in}\quad& &\Omega,\\
      u&>0\quad& &\textrm{in}\quad& &\Omega,\\
      u&=\kappa\mu& &\textrm{on}\quad& &\partial\Omega.
    \end{aligned}\right.
  \end{equation}
  Specifically, if $\Omega$ is an open subset of the cone $\{r\theta\in\R^N : r>0,\theta\in A\}$ with the exterior cone condition (see e.g., \cite{GT}*{Section 2.8}), the same results as in Theorems~\ref{Thm:Thm1'}~and~\ref{Thm:Thm2'} can be obtained for the problem \eqref{eq:GBP}.
  \end{Remark}
  \noindent
{\bf Acknowledgment.}
This work was supported by JSPS KAKENHI Grant Number 23KJ0645.

\end{document}